\documentclass[12pt]{iopartm}

\usepackage{color}
\usepackage{graphicx}
\usepackage{subfig}
\usepackage{iopams}
\usepackage{amsfonts,amsmath,amssymb,amsthm,amstext}
\usepackage{enumerate,enumitem,etex}
\usepackage{caption}

\usepackage{mathrsfs,bm}
\DeclareMathAlphabet{\mathscrbf}{OMS}{mdugm}{b}{d}

\usepackage[colorlinks,citecolor=blue,linkcolor= blue]{hyperref}

\newcommand{\Rbb}{\mathbb{R}}
\newcommand{\beps}{\bm{\varepsilon}}
\newcommand{\bvphi}{\bm{\varphi}}

\newcommand{\bCcal}{\bm{\mathcal{C}}}

\newcommand{\bIcal}{\bm{\mathcal{I}}}
\newcommand{\bI}{\bm{I}}
\newcommand{\bU}{\bm{U}}
\newcommand{\bu}{\bm{u}}

\newcommand{\buid}{\bm{u}_{1,\delta}}
\newcommand{\buiid}{\bm{u}_{2,\delta}}
\newcommand{\bui}{\bm{u}_{1}}
\newcommand{\buii}{\bm{u}_{2}}
\newcommand{\bv}{\bm{v}}
\newcommand{\bw}{\bm{w}}
\newcommand{\bff}{\bm{f}}
\newcommand{\bffg}{\bm{f}_{\!\gamma}}
\newcommand{\bfi}{\bm{f}_{\!1}}
\newcommand{\bfii}{\bm{f}_{\!2}}
\newcommand{\bg}{\bm{g}}
\newcommand{\be}{\bm{e}}
\newcommand{\ba}{\bm{a}}
\newcommand{\bb}{\bm{b}}
\newcommand{\bc}{\bm{c}}
\newcommand{\bd}{\bm{d}}
\newcommand{\bad}{\bm{a}_\delta}
\newcommand{\bbd}{\bm{b}_\delta}
\newcommand{\bcd}{\bm{c}_\delta}
\newcommand{\bdd}{\bm{d}_\delta}
\newcommand{\bfd}{\bm{f}_\delta}
\newcommand{\bfhd}{{\,\hat{\!\bm{f}}}_\delta}
\newcommand{\bh}{\bm{h}}
\newcommand{\bn}{\bm{n}}
\newcommand{\bx}{\bm{x}}
\newcommand{\bM}{\mathbf{M}}
\newcommand{\bDc}{\bm{\mathcal{D}}}
\newcommand{\bPc}{\bm{\mathcal{P}}}
\newcommand{\bLc}{\mathscrbf{L}}
\newcommand{\bLcd}{\mathscrbf{L}_\delta}
\newcommand{\bLcsd}{\mathscrbf{L}^*_\delta}
\newcommand{\Escr}{\mathscrbf{E}}
\newcommand{\bEcd}{\bm{\mathcal{E}}_\delta}
\newcommand{\nab}{\bm{\nabla}}
\newcommand{\nabs}{\bm{\nabla}^{\operatorname{s}}}
\newcommand{\hess}{\bm{\nabla}^{\otimes2}}
\newcommand{\bzero}{\bm{0}}
\newcommand{\bHc}{{\bm{\mathcal{H}}}}
\newcommand{\dc}{\!:\!}
\newcommand{\p}{\partial}
\newcommand{\trsp}{^{\mbox{\tiny$\mathsf{T}$}}}
\newcommand{\dev}{^{\mbox{\tiny D}}}
\newcommand{\spc}{\operatorname{sp}}
\newcommand{\td}{\text{d}}
\newcommand{\ado}{\alpha_\delta^o}
\newcommand{\ta}{\tilde{\alpha}}
\newcommand{\bdo}{\beta_\delta^o}
\newcommand{\tb}{\tilde{\beta}}
\newcommand{\dn}{_{\delta(n)}}

\newtheorem{theorem}{Theorem}
\newtheorem{proposition}{Proposition}
\newtheorem{lemma}{Lemma}

\newtheorem{remark}{Remark}
\newtheorem*{remark*}{Remark}

\newtheorem{Assumptions}{Assumptions}

\begin{document}

\title[Reconstruction of parameters in elasticity from noisy full-field measurements]{Reconstruction of constitutive parameters in isotropic linear elasticity from noisy full-field measurements}

\author{Guillaume Bal, C\'edric Bellis$\,^{1}$, S\'ebastien Imperiale$\,^{2}$, Fran\c{c}ois Monard$\,^{3}$} 
\address{Dept of Applied Physics and Applied Mathematics, Columbia University, New York NY, 10027, USA}
\address{Now at: $^1$ LMA, CNRS UPR 7051, 13402 Marseille Cedex 20, France}
\address{\hspace{14mm}$^2$ INRIA, M3DISIM Project-Team, 91120 Palaiseau, France}
\address{\hspace{14mm}$^3$ Dept of Mathematics, University of Washington, Seattle WA, 98195, USA}

\ead{gb2030@columbia.edu, bellis@lma.cnrs-mrs.fr, sebastien.imperiale@inria.fr, fmonard@math.washington.edu}


\begin{abstract}
Within the framework of linear elasticity we assume the availability of internal full-field measurements of the continuum deformations of a non-homogeneous isotropic solid. The aim is the quantitative reconstruction of the associated moduli. A simple gradient system for the sought constitutive parameters is derived algebraically from the momentum equation, whose coefficients are expressed in terms of the measured displacement fields and their spatial derivatives. Direct integration of this system is discussed to finally demonstrate the inexpediency of such an approach when dealing with noisy data. Upon using polluted measurements, an alternative variational formulation is deployed to invert for the physical parameters. Analysis of this latter inversion procedure provides existence and uniqueness results while the reconstruction stability with respect to the measurements is investigated. As the inversion procedure requires differentiating the measurements twice, a numerical differentiation scheme based on an ad hoc regularization then allows an optimally stable reconstruction of the sought moduli. Numerical results are included to illustrate and assess the performance of the overall approach.
\vspace{-2mm}
\end{abstract}


\section{Introduction}

The identification of constitutive parameters associated with deformable solids bears relevance to a wide range of applications \cite{BonnetIP} such as structural health monitoring, non-destructive material testing or medical imaging. A static or dynamical excitation applied to an elastic body gives rise to a measurable signature of such quantities of interest. The data, possibly noisy, might be then collected externally, i.e. at the boundary of the domain, or internally in the form of partial or full-field measurements. Classical identification methods aiming at the quantitative reconstruction of parameters revolve around the minimization of a given objective function. Such approaches typically suffer from requiring numerous forward solutions, and from an intrinsic ill-posedness yielding practical instabilities as discussed in e.g. \cite{Monard:phd}.

State-of-the-art experimental techniques \cite{Greenleaf,Parker,Grediac:Hild:FF,book:3D:imaging} offer a variety of non-invasive imaging modalities providing internal full-field measurements for (i) biological tissues using, e.g., ultrasound, magnetic resonance imaging or speckle interferometry, and (ii) materials by X-ray, neutron diffraction tomography or digital image correlation. Dedicated identification methods are also flourishing \cite{Parker,Avril2008}. The breakthrough is the availability of internal data, which has led to a paradigm shift in imaging: the constitutive parameters entering the model partial differential equation can now be reconstructed given the knowledge of one, or multiple, solutions of that PDE. In other words, the momentum equation is seen as a PDE for the unknown moduli with the measured displacement or strain fields constituting its spatially varying coefficients. This is the perspective adopted in \cite{McLaughlin} based on time-dependent data. In a similar context, the article \cite{Barbone2007} focuses on the case when one Lam\'e parameter is known {\it a priori} and only one measured displacement field is used to reconstruct the other one. In \cite{Barbone2010}, a dedicated variational formulation is employed to solve for the elasticity parameters using static displacement field measurements.


The approach adopted here is to be linked to the so-called {\em  hybrid} or {\em multi-wave} inverse problems \cite{Bal:rev:arxiv,Bal:inBook}, which under the assumption of a coupling between two physical phenomena, come as two-step inverse problems. A typical scenario is that a high-resolution imaging modality will provide {\em internal measurements} (e.g. the displacement fields, here) for a parameter-reconstruction problem involving an elliptic PDE. While this first step is assumed here, our starting point is the second step of this inversion. For such inverse problems where the forward model is an elliptic PDE, considering inversion from internal measurements (rather than, classically, measurements at the domain's boundary) greatly improve the mathematical ill-posedness, which practically implies great improvements in resolution on reconstructed images. As the problem remains mildly ill-posed, noisy measurements still require being regularized before the inversion step in order not to amplify high-frequency noise \cite{Engl}, though this regularization is much less stringent than when using boundary measurements, allowing recovery of information at smaller scales. \\

This study lies at the crossroads of experimental solid mechanics and hybrid inverse problems. It is similar in spirit to the so-called adjoint-weighted variational formulations introduced in \cite{Barbone:2008,Albocher:2009,Barbone2010}. However, advantage is taken here of an explicit formulation of a gradient system for the sought parameters. This approach stems from the method in \cite{Bal:inBook,Bal:CPAM,Monard:2012} for scalar diffusion equations, which is therefore extended here to the tensorial framework of elasticity. This enables further mathematical characterizations of the reconstruction stability and an inversion from a least squares approach. A regularization-based data differentiation scheme is also proposed to accommodate noisy measurements.

The article outline is as follows:
\begin{enumerate}[leftmargin=0cm,itemindent=8mm,labelwidth=\itemindent,labelsep=0cm,align=left,label=(\roman*),noitemsep,topsep=0pt] 
\item The governing equations of linear elasticity are presented in Section \ref{sec:prelim}. Section \ref{sec:rec:moduli} introduces the inverse problem of the quantitative reconstruction of the spatially-dependent constitutive moduli, namely the two eigenvalues of the elasticity tensor. Assuming the availability of a set of measured displacement fields, a system of PDEs for the unknown parameters is constructed by algebraic manipulations. Upon satisfying an invertibility condition, these equations are then recast into a simple gradient system.
\item In Section \ref{sec:ode}, this gradient system is discussed using a conventional ordinary differential equation-based approach which yields a local uniqueness and stability result.
\item Alternatively, when the available data is noisy, then a weak formulation of the least squares problem associated with the system at hand is presented in Section \ref{sec:varf}. Existence and uniqueness of a solution is finally shown using a standard analysis.
\item Strain and hessian tensors of measured displacement fields that enter the gradient system are clearly strongly detrimental to the inversion in the presence of noise. In this regard, the reconstruction stability in such configurations is analyzed in Section \ref{sec:stab:noise}.
\item Finally, a numerical differentiation scheme is proposed and analyzed in Section \ref{sec:num:reg}. It relies on a regularization of the measurements by $L^2$-projections on coarse, yet high-order, finite element spaces. The regularizing operator is constructed to enable an optimally stable reconstruction of the constitutive moduli. A set of numerical results is included in Section \ref{sec:num:res}.
\end{enumerate}

\section{Preliminaries}\label{sec:prelim}

Let $\Omega\subset \Rbb^d$ with $d=2$ or $3$, denote a regular enough elastic body which undergoes a time-harmonic infinitesimal transformation characterized by the displacement field $\bu$ and the frequency $\omega$. On noting $\p_{j}$ the $j$-th partial derivative and assuming that $|\p_{j} u_{i}|=o(1)$ for $1\leq i,j\leq d$, the corresponding linearized deformation is quantified by the second-order strain tensor $\beps=\nabs\bu:=\frac12[\nab \bu+(\nab\bu)\trsp]$. 
In the absence of any body force field, the displacement field satisfies the momentum equation \cite{Marsden} 
\begin{equation}
\nab\cdot(\bCcal\dc\beps)+\rho\,\omega^2\bu=\bzero \; \; \mbox{ in } \Omega,
\label{REFequ}
\end{equation}
where the mass density $\rho>0$ is assumed to be constant and with reference to the implicit time-harmonic factor $e^{\mathrm{i}\omega t}$. In the isotropic case considered, the elasticity tensor $\bCcal$ is characterized by only two eigenvalues $\alpha$ and $\beta$, see \cite{Knowles}, such that
\begin{equation}\label{decomp:C}
\bCcal=\frac{\alpha}{d}\bI\otimes\bI+\beta\Big(\bIcal_{\text{sym}}-\frac{1}{d}\bI\otimes\bI \Big),
\end{equation}
where $\bI$ and $\bIcal_\text{sym}$ are respectively the second-order and symmetric fourth-order identity tensors. In the ensuing analysis it will be seen that inverting for the elastic moduli $\alpha > 0$ and $\beta > 0$, rather than for the Lam\'e parameters $\lambda=(\alpha-\beta)/d$ and $\mu=\beta/2$, is facilitated by the decomposition and interpretation of the measured strain fields into elementary hydrostatic and deviatoric contributions. The analysis of the proposed reconstruction approach further requires that $(\alpha, \beta)$ are smooth enough, i.e.
\begin{Assumptions} \label{regularity_alpha} 
	The constitutive parameters satisfy: $ \alpha, \beta \; \in \; L^\infty(\Omega) \cap H^1(\Omega). $
\end{Assumptions} 
\noindent The boundedness assumption is a direct consequence of standard thermo-mechanical stability conditions \cite{Mehrabadi:62}. The requirement $\alpha, \beta \in H^1(\Omega)$ is more specifically associated with the $H^1(\Omega)$-norm that is employed hereinafter to measure the parameters reconstruction quality. Yet, some regularity on the solution of \eref{REFequ} is required, in general, and this is to be achieved with sufficiently smooth coefficients.\\ 

To be used in the ensuing developments, for $p=2$ or $p=\infty$ with associated Lebesgue space $L^p(\Omega)$, and $(m_1,m_2)\in \mathbb{N}^2$, let us introduce a norm associated with matrices $\mathbf{H}\in L^p(\Omega)^{m_1\times m_2}$ as
\begin{equation}\label{def:norm:mat}
  \| \mathbf{H} \|_{L^p} :=\|\, | \mathbf{H} |\, \|_{L^p},
\end{equation}
which is expressed in term of the Frobenius norm $|\mathbf{H}|=(\mathbf{H}\dc\mathbf{H})^{1/2}$ with double inner product ``$\,\dc\,$''. Moreover, given $\ell\geq0$, Sobolev spaces $W^{\ell,p}(\Omega)$ and $H^{\ell}(\Omega)=W^{\ell,2}(\Omega)$, then for all vectors $\bh\in W^{\ell,p}(\Omega)^{m_1}$, the definition \eref{def:norm:mat} remains valid with $m_2=1$ when $\ell=0$ by replacing the Frobenius norm by the Euclidean norm $|\bh|=(\bh\cdot\bh)^{1/2}$ that uses the single inner product. When $\ell\geq1$, given an orthonormal basis $\{\be_i\}_{i}$ of $\Rbb^d$, one defines
\[
\| \bh \|_{H^{\ell}}=\Bigg(\sum_{i=1}^{m_1}\|  \bh\cdot\be_i \|^2_{H^{\ell}}\Bigg)^{\!\frac{1}{2}}\quad\text{and}\quad \| \bh \|_{W^{\ell,\infty}}=\sum_{i=1}^{m_1}\|  \bh\cdot\be_i \|_{W^{\ell,\infty}}.
\]

\section{Reconstruction of elastic moduli}\label{sec:rec:moduli}

This study focuses on the reconstruction of the spatially varying elastic moduli $\alpha$, $\beta$ from the knowledge of two displacement fields $\bu_{1}$, $\bu_{2}$ within $\Omega$. The choice of this measurement set is supported by earlier studies \cite{Barbone2010,Bal:rev:arxiv,Bal:inBook}. For $n=1,2$, one denotes strain tensors by $\beps_n=\nabs\bu_n$ with trace and deviatoric counterpart given by
\begin{equation}\label{eq:classical_def}
t_n=\tr(\beps_{n}), \quad	\beps\dev_{n}=\beps_{n}-\frac{t_{n}}{d}\bI.
\end{equation}

\subsection{Inversion formula}

Plugging the decomposition \eref{decomp:C} into the momentum equation \eref{REFequ} one obtains
\begin{equation}\label{eq:sol_general}
\nab\cdot(\bCcal:\beps_n) + \rho\,\omega^2_n \bu_n = \nab\cdot \Big(\alpha\frac{t_n}{d}  \bI + \beta \beps_n\dev \Big) + \rho\,\omega^2_n \bu_n=\bzero, \qquad n=1,\,2,
\end{equation}
which, expanding using the product rule, yields
\begin{equation}\label{exp:hooke}
    \frac{t_n}{d} \nab\alpha + \alpha \nab\frac{t_n}{d} + \beps_n\dev\cdot\nab\beta + \beta\nab\cdot\beps_n\dev + \rho\,\omega^2_n \bu_n = \bzero, \qquad n=1,\,2.
\end{equation}
On introducing the matrices
\begin{equation}\label{eq:first_def_b}
\mathbf{A}_{n}= \left[\!\begin{array}{cc}\frac{t_{n}}{d}\bI & \beps_{n}\dev\end{array}\!\right]\in\Rbb^{d\times 2d}
\quad\text{and}\quad
\mathbf{B}_{n}=\left[\!
\begin{array}{cc}
\nab\frac{t_n}{d}  &  \nab\cdot\beps_{n}\dev
\end{array}\!\right]\in\Rbb^{d\times 2},
\end{equation}
then Eqn. \eref{exp:hooke} reads
\[
\mathbf{A}_{n}\left[\!\begin{array}{c} \nab \alpha \\ \nab\beta\end{array}\!\right]+\mathbf{B}_{n}\left[\!\begin{array}{c} \alpha \\ \beta\end{array}\!\right] = -\rho\,\omega_n^2 \bu_n, \qquad n=1,2.
\]
We then combine these equations into an overdetermined PDE system of the form 
\begin{equation}\label{MBdef}
  \mathbf{A}\left[\!\begin{array}{c} \nab \alpha \\ \nab\beta\end{array}\!\right]+\mathbf{B}\left[\!\begin{array}{c} \alpha \\ \beta\end{array}\!\right]= - \left[\!\begin{array}{c} \rho\,\omega_1^2 \bu_{1} \\  \rho\,\omega_2^2 \bu_2 \end{array}\!\right], \quad\text{where} \quad \mathbf{A} := \left[\!\begin{array}{c} \mathbf{A}_1 \\ \mathbf{A}_2 \end{array}\!\right] \quad \text{and} \quad \mathbf{B} := \left[\!\begin{array}{c} \mathbf{B}_1 \\ \mathbf{B}_2 \end{array}\!\right].
\end{equation}
We now make more precise when and how to invert the $2d\times 2d$ matrix-valued $\mathbf{A}$. 
\begin{lemma}\label{Ainv}
    Let us define 
\begin{equation}\label{eq:def_E}
	\Escr:=t_{1}\beps_{2}\dev-t_{2}\beps_{1}\dev
\end{equation} 
Suppose that  $\Escr$ is uniformly invertible, then $\mathbf{A}$ defined in \eref{MBdef} is uniformly invertible and  
\[
\mathbf{A}^{-1}=\left[\!
\begin{array}{cc}
d\,\beps_{2}\dev  &  -d\,\beps_{1}\dev\\
-{t_{2}}\bI  &  {t_{1}}\bI\\ 
\end{array}\!
\right]
\left[\!
\begin{array}{cc}
\Escr^{-1}  &  \bzero\\
\bzero &  \Escr^{-1}\\ 
\end{array}\!
\right].
\]
\end{lemma}
Using Lemma \ref{Ainv} and upon inverting the block matrix $\mathbf{A}$, one arrives at the system 
\begin{align}
    \bLc(\alpha,\beta) = \bff, \qquad\text{where}\qquad \bLc(\alpha,\beta) := \left[\!\begin{array}{c} \nab \alpha \\ \nab\beta\end{array}\!\right] + \mathbf{M} \left[\!\begin{array}{c} \alpha \\ \beta\end{array}\!\right]
    \label{sys:op:ref}
\end{align}
with
\begin{equation}\label{def:M:f}
\mathbf{M}=\left[\!\begin{array}{cc} \ba & \bb \\ \bc & \bd  \end{array}\!\right]:=\mathbf{A}^{-1}\,\mathbf{B}\quad\text{and}\quad \bff=\left[\!\begin{array}{c} \bfi\\ \bfii \end{array}\!\right] := - \mathbf{A}^{-1}\left[\!\begin{array}{c} \rho\,\omega_1^2 \bu_{1} \\  \rho\,\omega_2^2 \bu_2 \end{array}\!\right].
\end{equation}
When  $\mathbf{M} \in L^\infty(\Omega)^{2d\times 2} $, the operator $\bLc\!$ is linear with values from $H^1(\Omega)^2$ into $L^2(\Omega)^{2d}$. Note that such a system is valid both in 2D and 3D upon substitution of the correct expression for the strain tensors and their traces. Moreover, one is required to verify that the tensor $\Escr$ is invertible.

The elasticity parameters can be reconstructed from Eqn. \eref{sys:op:ref} either by direct integration of this gradient system, which requires some compatibility conditions, or using a least squares based variational formulation. These two approaches will be addressed in sections \ref{sec:ode} and \ref{sec:varf} respectively. In any case, one can already deduce some regularity requirement so that \eref{sys:op:ref} make sense in a standard functional sense for which $\mathbf{M} \in L^\infty(\Omega)^{2d\times 2} $ and $ \bff \in L^\infty(\Omega)^{2d}$. These hypotheses are summarized below:

\begin{Assumptions}\label{mega_hyp}
The displacement field solutions satisfy $\bu_n \in H^2(\Omega)^d \cap W^{1,\infty}(\Omega)^d$ for $n=1,2$ and $\inf_{\Omega} (| \det \Escr \, |) \geq c_0 > 0$.
\end{Assumptions}
\begin{remark}
It is generally assumed that incompressible materials, such as biological soft tissues \cite{Fung} can be thought of as the limit of the compressible case in the limit $\alpha\to\infty$. In this case, $\tr(\beps) \to 0$ which implies $ \det\Escr \to0$ so that Assumptions \ref{mega_hyp} are violated. Therefore, the compressible and the incompressible cases are very dissimilar, the latter being beyond the scope of this article. Reference to \cite{Barbone2007}, Section 8, can be made for a discussion on the discrepancy between these two cases.
\end{remark}

\subsection{Invertibility conditions}\label{sec:invert:cond}

As seen in the section above, the reconstruction procedure requires the tensor $\Escr$ to be invertible. This section now aims at clarifying this assumption.

\begin{proposition} The tensor $\Escr:=t_{1}\beps_{2}\dev-t_{2}\beps_{1}\dev$ is non-invertible if and only if the strain tensors satisfy $t_1 \beps_{2} = t_2\beps_{1}$ when $d=2$, or $t_1 \beps_{2} = t_2\beps_{1} + \gamma(\bphi_{1}\otimes\bphi_{1}-\bphi_{2}\otimes\bphi_{2})$ when $d=3$ for some scalar $\gamma$ and two unit orthogonal vectors $\bphi_1$, $\bphi_2$. 
\end{proposition}
\noindent In particular, $\Escr$ is non-invertible if, given strain tensor $\beps_{1}$, there exists $\kappa\in\Rbb\setminus\{0\}$ such that $\beps_2$ satisfies $\beps_{2}=\kappa\,\beps_1$ (2D) or $\beps_{2}=\kappa\,\beps_1+ t_1^{-1}\gamma(\bphi_{1}\otimes\bphi_{1}-\bphi_{2}\otimes\bphi_{2})$ (3D).

\begin{proof}[Proof of Proposition 1.] By construction, the tensor $\Escr=t_{1}\beps_{2}\dev-t_{2}\beps_{1}\dev=t_{1}\beps_{2}-t_{2}\beps_{1}$ is symmetric and traceless. In two dimensions, this implies that it takes the form $\Escr=\left[ \begin{smallmatrix} a & b \\ b & -a \end{smallmatrix} \right]$ with $a,b$ two real parameters, in which case $\det \Escr = - (a^2 + b^2)$, so $\det\Escr = 0$ is equivalent to $\Escr = \bzero$. In the case $d=3$ one has 
$
\spc(\Escr)=\{\gamma_{1},\gamma_{2},-(\gamma_{1}+\gamma_{2})\}.
$
with $\gamma_{1},\,\gamma_{2}\in\Rbb$. The condition $\det(\Escr)=0$ entails $\gamma_{i}=0$ for $i\in\{1,2\}$ or $\gamma_{2}=-\gamma_{1}$ so that one can conclude to the existence of an orthonormal basis $\{\bphi_{i}\}_{i}$ of $\Rbb^d$ such that $\Escr=\gamma(\bphi_{1}\otimes\bphi_{1}-\bphi_{2}\otimes\bphi_{2})$.
\end{proof}
 
\paragraph{How to achieve condition $\bm{\det\Escr\!\ne\! 0}$ ?} Solutions $\bu_1, \bu_2$ have to be constructed such that $\det \Escr\ne 0$ is achieved locally or globally. For example, a strategy is to generate displacement fields satisfying the properties that (i) $t_1 = 0$, (ii) $\beps_1$ invertible and (iii) $t_2\ne 0$, since in this case the tensor $\Escr$ takes the form $\Escr = -t_2 \beps_1$. From the practical standpoint, the remaining question concerns the control of these solutions from the domain boundary as discussed below.

\textit{Constant coefficients:} A simple prototypical example in the case of constant coefficients and stationary regime is to choose, in dimension $d=2$ or $ d=3$, 
\begin{equation}\label{sol:cst:coef}
    \bu_1(\bx)\cdot \be_j = \sum_{i\ne j} x_i, \quad 1\leq j\leq d, \qquad \text{and} \quad \bu_2(\bx) = \bx.
\end{equation}
In this case, we have that $\beps_1 = \bU - \bI$, where $\bU$ is an $d\times d$ matrix filled with ones, and $\beps_2 = \bI$. Here we globally have $t_1 = 0$, $t_2 = d\ne 0$ and $\det\beps_1 = (-1)^{d-1} (d-1)\ne 0$, so that $\Escr = t_1\beps_2-t_2\beps_1$ is nowhere singular. Note that leeway is allowed in choosing admissible fields $ \bu_1$, $ \bu_2$ and known analytical solutions of elasticity can be employed.

\textit{Non-constant coefficients:} In the case of non-constant coefficients, it is no longer clear theoretically that one can achieve the condition $\det \Escr\ne 0$ globally with only two solutions, although numerical results in Section \ref{sec:num:res} will show that numerically, two solutions, chosen after the example above for instance, are usually enough to obtain satisfactory global reconstructions. 

Following ideas initiated by Bal and Uhlmann in \cite{Bal:CPAM} (see also \cite{Bal2013,Monard2012a,Bal2013a}), let us mention that under some regularity assumptions, it can be proven that there exist boundary conditions such that condition $\det\Escr\ne 0$ holds locally. As seen in the scalar case, such proof will not construct the boundary conditions explicitly. Yet it justifies that the nonvanishing determinant condition can be achieved and provides a way to prove local reconstructibility of parameters from full-field measurements. The main tool for such proofs is the Runge approximation property, which itself follows from unique continuation property, established in e.g. \cite{Nakamura2005,Lin2009,Lin2011} in the context of elasticity.

\section{ODE-based approach and stability estimate}\label{sec:ode}

In this section we consider a direct integration approach of the gradient system \eref{sys:op:ref}, its main interest (stability estimate) and its main shortcoming (dependence on integration curve). Considering a connected subset $\Omega_0\subseteq\Omega$ and two points $\bx_0, \bx \in \Omega_0$, let 
\[
\gamma: [0,1]\ni t\mapsto \bxi(t)\quad \text{such that }\quad \bxi(0)=\bx_0 \text{ and } \bxi(1)=\bx,
\]
be a smooth curve with endpoints $\bx_0$ and $\bx$. Using the chain rule along this curve, the gradient system \eref{sys:op:ref} yields an ODE along the curve $\gamma$, of the form 
\begin{equation}\label{ode:ref}
\frac{\td \bvphi_\gamma(t)}{\td t} + \bM_\gamma(t) \bvphi_\gamma(t)= \bffg(t) \  \  \text{with}\  \   \bvphi_\gamma(t):=\left[\!\begin{array}{c}\alpha\circ \gamma(t) \\ \beta\circ \gamma(t)\end{array}\!\right]\  \   \text{and}\  \    \bvphi_\gamma(0)=\left[\!\begin{array}{c}\alpha(\bx_0) \\ \beta(\bx_0)\end{array}\!\right]
\end{equation}
and where one has defined, for $t\in [0,1]$
    \[
\bM_\gamma(t) := \left[\!
	\begin{array}{cc}
	     \dot\bxi(t)\cdot\ba\circ \gamma(t) & \dot\bxi(t)\cdot\bb\circ \gamma(t) \\ \dot\bxi(t)\cdot\bc\circ \gamma(t) & \dot\bxi(t)\cdot\bd\circ \gamma(t)
	\end{array}\!
	\right] \quad \text{and}\quad \bffg(t) := \left[\!\begin{array}{c} \dot\bxi(t)\cdot\bfi\circ \gamma(t)  \\  \dot\bxi(t)\cdot\bfii\circ \gamma(t) \end{array}\!\right].
\]

From the values $(\alpha(\bx_0),\beta(\bx_0))$, the values $(\alpha(\bx),\beta(\bx))$ can be computed via direct integration of ODE \eref{ode:ref}. Fixing $\bx_0$ and varying $\bx$ throughout $\Omega_0$, this induces a reconstruction procedure of $(\alpha,\beta)$ throughout $\Omega_0$, since $\Omega_0$ is connected, every $\bx\in \Omega_0$ can be connected to $\bx_0$ via a smooth curve. This approach yields a unique and stable reconstruction of the sought moduli in the sense of the following theorem. 
\begin{theorem}\label{thm:localstab} 
  Consider a convex subset $\Omega_0\subseteq\Omega$ and two data sets $(\bu_1, \bu_2)$, $(\bu_1', \bu_2')$ with $\mathcal{C}^2$-smooth components, associated to two set of parameters $(\alpha,\beta)$ and $(\alpha',\beta')$ and satisfying $\inf_{\Omega_0} (| \det \Escr \, |, | \det \Escr' \, |) \geq c_0 >0$. Then the data sets determine uniquely the parameters $(\alpha,\beta)$ and $(\alpha',\beta')$  over $\Omega_0$ and there exists a positive constant $C$ such that
  \[
    \|\alpha - \alpha'\|_{W^{1,\infty}(\Omega_0)} + \|\beta-\beta'\|_{W^{1,\infty}(\Omega_0)} \leq C\Big( \epsilon_0 + \sum_{n=1}^2 \|\bu_n - \bu_n'\|_{W^{2,\infty}(\Omega_0)}  \Big),
    \]
  where $\epsilon_0 = |\alpha(\bx_0)-\alpha'(\bx_0)| + |\beta(\bx_0) - \beta'(\bx_0)|$ is the error committed at $\bx_0\in\Omega_0$.  
\end{theorem}

The proof of Theorem \ref{thm:localstab} relies on the Picard-Lindel\"of theorem and Gronwall's lemma to control the propagation of errors along characteristic curves, similar analysis can be found in coupled-physics contexts in e.g. \cite{Monard:2012} and is omitted here. While this approach provides us with an explicit reconstruction procedure yielding unique and stable reconstructions, this uniqueness depends on the choice of curve joining $\bx_0$ and $\bx$. This is because the gradient system \eref{sys:op:ref} is redundant and requires that the measurements fulfill additional integrability conditions in order for the reconstruction to not depend on this choice of integration curve. Noise in the measurements will however violate these conditions, thereby making this reconstruction approach somewhat not as well-adapted to noisy measurements as the variational formulation that follows.

\section{Variational formulation}\label{sec:varf}

Hereafter, the regularity requirements in Theorem \ref{thm:localstab} are relaxed and Assumptions \ref{mega_hyp} are considered in the ensuing developments.

\subsection{Noisy data and least squares approach}\label{sec:noisy:data}

Following the above discussion, from now on, let $\{(\buid, \buiid)\}$ denote a set of noisy measurements where the real-valued index parameter $\delta\geq0$ is intended to represent a \emph{measure} of the noise. In such a case, direct differentiations of these measurements, as required in \eref{sys:op:ref}, might amplify noise at high spatial frequencies thereby preventing a successful identification. Therefore, the first step towards the formulation of a reconstruction algorithm is to introduce some approximations, in a sense that will be specified later on, of the strain and hessian tensors of the noisy displacements fields as 
\begin{equation}\label{eq:approx}
		\beps_{n, \delta} \, \sim \,  \beps_n\quad \text{and} \quad  \bHc_{n, \delta} \, \sim \, \sum_{k=1}^d\hess (\bu_{n,\delta}\cdot\be_k)\otimes \be_k  \quad \text{for } n=1,2.
\end{equation}
with the short-hand notation
\begin{equation}\label{def:hess}
\bHc_{n, \delta}^k=\bHc_{n, \delta}\cdot\be_k\quad \text{for }k=1,\dots,d \quad \text{so that} \quad   (\mathcal{H}^k_{n,\delta})_{ij} \sim \partial_{i}  \partial_{j}  (\bu_{n,\delta}\cdot\be_k).
\end{equation}
The explicit constructions of the second and third-order tensors $ \beps_{1, \delta}, \beps_{2, \delta}$ and $\bHc_{1, \delta}, \bHc_{2, \delta}$ from the noisy data $ (\buid, \buiid) $ is the focus of Section \ref{sec:num:reg}. For now, we consider the noisy operator $ \bLcd $, featuring the noisy matrix $\mathbf{M}_\delta $ with components $\bad$, $\bbd$, $\bcd$ and $\bdd$, together with the noisy source term $ \bfd $ that are formally obtained upon substitutions of
\begin{gather*}
\bu_n \leftarrow \bu_{n,\delta}, \\ t_n \leftarrow t_{n,\delta}=\tr(\beps_{n,\delta}), \quad	\beps\dev_n\leftarrow \beps\dev_{n,\delta}=\beps_{n,\delta}-\frac{t_{n,\delta}}{d}\bI, \quad  \Escr\leftarrow \Escr_\delta:=t_{1,\delta}\beps_{2,\delta}\dev-t_{2,\delta}\beps_{1,\delta}\dev, \\
	\nab t_n \leftarrow   \sum_{i,j=1}^d   (\mathcal{H}_{n, \delta}^j)_{ij} \, \be_i ,  \quad  \nab\cdot\beps_{n} \leftarrow \sum_{i,j=1}^d \left( \frac{d-2}{2\, d} (\mathcal{H}_{n, \delta}^j)_{ij} + \frac{1}{2} (\mathcal{H}_{n, \delta}^i)_{jj}  \right)\be_i.
	\end{gather*}
	for $n=1,2$ into the equations from \eref{eq:first_def_b} to \eref{def:M:f}. With the noisy operator $\bLcd$ at hand, the problem \eref{sys:op:ref} is correspondingly recast as the minimization problem 
\begin{equation}\label{sys:op:ref:d}
	\inf_{ (\alpha_\delta,\beta_\delta) } \| \bLcd (\alpha_\delta,\beta_\delta)  - \bfd \|_{L^2(\Omega)}^2,  \qquad \text{ for all }\delta>0.
\end{equation}
When $\delta=0$, we define $\bLc_0\equiv\bLc$ and $\bff_{\!0}\equiv\bff$ with the corresponding solution to \eref{sys:op:ref:d} satisfying $(\alpha_0,\beta_0)\equiv(\alpha,\beta)$ and it will be proven next by convergence analysis that this solution coincides with the limit solution to the perturbed problem series \eref{sys:op:ref:d} as $\delta\to0$.

We first use a lifting $(\alpha^\ell,\beta^\ell)$ of the assumed-to-be-known boundary values of $(\alpha,\beta)$
\begin{equation}\label{lifting_notation}
	\alpha_\delta=\ado+\alpha^\ell \quad \text{and} \quad \beta_\delta=\bdo+\beta^\ell,
\end{equation}
to pose a problem in $(\ado,\bdo)\in H^1_0(\Omega)^2$. In order to compute the solution of the minimization problem \eref{sys:op:ref:d}, see \cite{Bochev}, we consider solving the associated equation 

\begin{equation}
    \bLcsd \bLcd(\ado,\bdo) = \bLcsd\big(\bfd-\bLcd (\alpha^\ell,\beta^\ell)\big),    
    \label{eq:normal}
\end{equation}
 where $\bLcsd:L^2(\Omega)^{2d}\to H^{-1}(\Omega)^2$ denotes the adjoint of $\bLcd:H^1_0(\Omega)^2\to L^2(\Omega)^{2d}$ with $H^{-1}(\Omega)$ the dual space to $H_0^1(\Omega)$. Formal manipulations yield
\[
    \bLcsd\,\bh  = - \left[\!\begin{array}{c} \nab\cdot\bh_1 \\ \nab\cdot\bh_2  \end{array}\!\right] + {\mathbf{M}_\delta}^{\!\!\raisebox{0.5ex}{\tiny$\mathsf{T}$}} \left[\!\begin{array}{c}  \bh_1 \\  \bh_2 \end{array}\!\right], \qquad \forall\,\bh=(\bh_1, \bh_2) \in L^2(\Omega)^{2d}.
\]
One can check by application of Green's formula that for all $(\ado,\bdo)$ and $(\ta,\tb)$ in $H^1_0(\Omega)^2$
\begin{equation}\label{VF}
	\big\langle \bLcsd\bLcd(\ado,\bdo) ,(\ta,\tb) \big\rangle_{L^2(\Omega)} = \big\langle \bLcd(\ado,\bdo) ,\bLcd(\ta,\tb) \big\rangle_{L^2(\Omega)}.
\end{equation}

\begin{remark}
	The knowledge of the boundary values of  $\alpha$ and  $\beta$ over the entire boundary $\partial \Omega$ represents redundant data. However, such assumption simplifies the following analysis since only \emph{essential} boundary conditions, i.e. Dirichlet boundary conditions, can be employed in the variational formulation of the problem. When $ (\alpha, \beta) $ is known only in a subdomain $\partial \Omega_0\subseteq\partial\Omega$ then the gradient equation \eref{sys:op:ref} can be used to deduce \emph{natural} boundary conditions, in terms of fluxes $ \nab \alpha \cdot \bn $ and $ \nab \beta \cdot \bn $ on $\partial \Omega\setminus\partial\Omega_0$, which would then appear in the integration by parts \eref{VF}.
\end{remark}

Finally, note that an alternative to solving system \eref{eq:normal} is the {Adjoint-Weighted variational Equation method} presented in \cite{Barbone2010}. It is based on a variational formulation that features a \emph{weighting} operator as a substitute for the adjoint operator $\bLcsd$ in \eref{eq:normal}.

\subsection{Existence and Uniqueness}

In what follows we study existence, uniqueness and continuity estimates for the solutions to Eqn. \eref{eq:normal}. Although this analysis is fairly standard when the noise parameter $\delta$ is fixed, it remains non trivial to obtain continuity estimates that are independent of  $\delta$. This is the focus of the following developments. To this aim, let us express the following hypotheses that describe the assumed uniform boundedness w.r.t. $\delta$ of the noisy coefficients $(\bad,\, \bbd,\, \bcd,\, \bdd)$ featured in Sec. \ref{sec:noisy:data}.
\begin{Assumptions}\label{hyp:LB}
	There exists $B>0$ such that for all $\delta>0$ one has\\ \centerline{$\bad,\, \bbd,\, \bcd,\, \bdd \in L^{\infty}_B(\Omega)^d:=\big\{\bh\in L^\infty(\Omega)^d, \, \|\bh\|_{L^\infty(\Omega)}\leq B\big\}$.}  
\end{Assumptions}

We first prove a uniqueness results associated with the operator $\bLcsd\bLcd$.

\begin{lemma}\label{eq:injnormal}
    If Assumptions \ref{hyp:LB} hold, then for all $\delta>0$ the operator $\bLcsd\bLcd: H_0^1(\Omega)^2 \to H^{-1}(\Omega)^2$ is injective. 
\end{lemma}

\begin{proof}
    If $\bLcsd\bLcd (\ta,\tb) = \bzero$ for some $(\ta,\tb) \in H^1_0(\Omega)^2$ then multiplying by $(\ta,\tb)$, integrating over $\Omega$ and integrating by parts yield
    \begin{equation}\label{norm:opA}
	\|\bLcd(\ta,\tb)\|^2_{L^2(\Omega)}=\int_{\Omega} \big\{ |\nab\ta + \bad \ta + \bbd \tb|^2 + |\nab\tb + \bcd\ta + \bdd \tb|^2 \big\} \,\td \bx = 0, 
    \end{equation}
    thus we have $\bLcd(\ta,\tb) = \bzero$ throughout $\Omega$, which implies
	\[
		|\nab\ta | + |\nab\tb| \leq C \, \big( \,  |\ta| + |\tb|  \,  \big) \quad \text{in } \Omega,
	\]
	where $C$ is a constant that depends on $B$. Then either an approach similar in spirit to establishing uniqueness in Theorem \ref{thm:localstab}, or the unique continuation principle, see for instance \cite[Lemma 8.5]{Colton} adapted to gradient equations and $H^1_0(\Omega)$ solutions, allows to conclude that $(\ta,\tb)\equiv (0,0)$ throughout $\Omega$. 
\end{proof}

Next, we show an inequality of Poincar\'e-Friedrichs type which is of key importance for the ensuing developments and numerical schemes. 
\begin{proposition}\label{prop:coer}
Provided that Assumptions \ref{hyp:LB} are satisfied, then there exits $C>0$ depending on $B$ but not on $\delta$ such that for all $\delta>0$ one has
\begin{equation}\label{ineq:PF}
\|\ta\|_{H^1_0(\Omega)}^2 + \|\tb\|_{H^1_0(\Omega)}^2 \leq C\|\bLcd(\ta,\tb)\|_{L^2(\Omega)}^2, \qquad \forall(\ta,\tb)\in H^1_0(\Omega)^2.
\end{equation}
\end{proposition}

\begin{proof}
By contradiction, assume that, for any $\delta\geq0$ the inequality \eref{ineq:PF} does not hold, i.e.  one can construct a sequence  
\[
\big\{(\ta\dn,\tb\dn,\ba\dn,\bb\dn,\bc\dn,\bd\dn)\big\}_{n\geq 1} \in H^1_0(\Omega)^2\times \Big[L^\infty_B(\Omega)^d\Big]^4, 
\]
such that $\|\ta\dn\|_{H^1_0(\Omega)}^2 + \|\tb\dn\|_{H^1_0(\Omega)}^2 > C \, n \, \|\bLc\dn(\ta\dn,\tb\dn)\|_{L^2(\Omega)}^2$.\\

After re-normalization and without loss of generality we can assume that
\[
\text{(i)}\  \   \| \ta\dn \|_{\textcolor{black}{H^1_0(\Omega)}}^2 + \| \tb\dn \|_{\textcolor{black}{H^1_0(\Omega)}}^2=1,\qquad\qquad\text{(ii)} \  \   \big\|\bLc\dn\big(\ta\dn,\tb\dn\big)\big\|_{L^2(\Omega)}^2 < \frac{1}{n}  .
\]
From (i), the sequence $\big\{(\ta\dn,\tb\dn)\big\}$ is bounded in the Hilbert space $H^1_0(\Omega)^2$ and by definition of the Banach space $L^\infty_B(\Omega)$, there exist $\alpha_{*},\beta_{*}$ in $H^1_0(\Omega)$ and $\ba_*,\bb_*,\bc_*,\bd_*\in L^{\infty}_B(\Omega)^d$ such that up to subsequences one has
\begin{equation}
\begin{aligned}
& (\ta\dn,\tb\dn)\underset{H^1_0}{\rightharpoonup}(\alpha_*,\beta_*)\text{ and }(\ta\dn,\tb\dn)\underset{L^2}{\rightarrow}(\alpha_*,\beta_*),\\[1mm]
& (\ba\dn,\bb\dn,\bc\dn,\bd\dn)\underset{L^\infty}{\overset{*}{\rightharpoonup}}(\ba_*,\bb_*,\bc_*,\bd_*).
\end{aligned}
\label{seq:lim}
\end{equation}
Next, one seeks to assess the weak convergence limit of $\big\{\bLc\dn\big(\ta\dn,\tb\dn\big)\big\}$, i.e. whether
\[
    \big\langle\bLc\dn\big(\ta\dn,\tb\dn\big),(\bv,\bw)\big\rangle_{L^2(\Omega)} \underset{n\to+\infty}{\longrightarrow} \big\langle\bLc_*\big(\alpha_*,\beta_*\big),(\bv,\bw)\big\rangle_{L^2(\Omega)},\quad  \forall (\bv,\bw)\in L^2(\Omega)^{2d}.
\]
where the operator $\bLc_*$ is defined as in \eref{sys:op:ref} in terms of $\ba_*,\bb_*,\bc_*$ and $\bd_*$.
Given that 
\[\begin{aligned}
\big\langle\bLc\dn\big(\ta\dn,\tb\dn\big)-\bLc_*\big(\alpha_*,\beta_*\big),(\bv,\bw)\big\rangle_{L^2(\Omega)}=\hspace*{6.1cm}\\
+\int_{\Omega}\left\{ \nab (\ta\dn-\alpha_*)\cdot \bv+\nab  (\tb\dn-\beta_*)   \cdot \bw \right\}\td \bx\\
+\int_{\Omega}\left\{ \big[(\ba\dn-\ba_*)\cdot\bv+(\bc\dn-\bc_*)\cdot\bw\big]\alpha_*+\big[(\bb\dn-\bb_*)\cdot\bv+(\bd\dn-\bd_*)\cdot\bw\big]\beta_* \right\}\td \bx \\
+\int_{\Omega}\left\{ (\ba\dn\cdot\bv+\bc\dn\cdot\bw)(\ta\dn-\alpha_*)+(\bb\dn\cdot\bv+\bd\dn\cdot\bw)(\tb\dn-\beta_*) \right\}\td \bx.
\end{aligned}\]
Owing to the respective convergences \eref{seq:lim} of the right-hand side terms in the previous equation, i.e. weak convergence in $H^1_0$, weak-$*$ in $L^\infty$ and  strong in $L^2$, one can conclude that $ \bLc\dn\big(\ta\dn,\tb\dn\big) $ converges weakly in $L^2$ to $ \bLc_*\big(\alpha_*,\beta_*\big)$. This implies
\[
	\|  \bLc_*\big(\alpha_*,\beta_*\big) \|_{L^2(\Omega)} \leq \underset{n \rightarrow +\infty}{\lim \inf} \,  	\|   \bLc\dn\big(\ta\dn,\tb\dn\big)   \|_{L^2(\Omega)} = 0.
\]
Finally, one has $\bLc_*\big(\alpha_*,\beta_*\big)=\bzero$ which using Lemma \ref{eq:injnormal} yields $\alpha_*=\beta_*=0$.  Remark that
\begin{multline*}
	\frac{1}{n} \geq \| \bLc\dn\big(\ta\dn,\tb\dn\big) \|_{L^2(\Omega)}^2 \geq 
	\frac{1}{2}\,  \big( \, \| \ta\dn \|_{ {H^1_0(\Omega)}}^2 +  \| \tb\dn \|_{ {H^1_0(\Omega)}}^2 \, \big) \\ -   \, \big( \, \| \ba\dn \ta\dn + \bb\dn \tb\dn \|_{L^2(\Omega)}^2 + \|  \bc\dn \ta\dn + \bd\dn \tb\dn  \|_{L^2(\Omega)}^2 \,\big),
\end{multline*}
and so, using the strong convergence in $L^2$:
\[
4 \, B^2 \, \left(\| \ta\dn \|_{L^2(\Omega)}^2 + \| \tb\dn \|_{L^2(\Omega)}^2 \right) \geq \frac{1}{2}-\frac{ 1 }{ n} \; \Rightarrow  \; \| \ta_* \|_{L^2(\Omega)}^2 + \| \tb_* \|_{L^2(\Omega)}^2 \geq \frac{1}{4  \,  B^2},
\]
which contradicts  $\alpha_*=\beta_*=0$.
\end{proof}

With Proposition \ref{prop:coer} at hand, we are now in position to prove existence and uniqueness of a solution to Eqn. \eref{eq:normal}. Given data satisfying Assumptions \ref{hyp:LB}, finding $(\ado,\bdo)\in H^1_0(\Omega)^2$ solving \eref{eq:normal} is equivalent to solving the weak form
\begin{equation}\label{weak:form}
\! \big\langle \bLcd(\ado,\bdo) , \bLcd(\ta,\tb) \big\rangle_{L^2(\Omega)} = \big\langle \bfd-\bLcd (\alpha^\ell,\beta^\ell) , \bLcd(\ta,\tb) \big\rangle_{L^2(\Omega)}, \;  \forall(\ta,\tb)\in H^1_0(\Omega)^2.
\end{equation}
From the Cauchy-Schwarz inequality, \eref{norm:opA} and the definition of the space $L^{\infty}_B(\Omega)$ then there is a constant $C'>0$ such that
\begin{equation}\label{bound:opA}
\big\langle \bLcd(\ado,\bdo) , \bLcd(\ta,\tb) \big\rangle_{L^2(\Omega)} \leq  C' \big(\|Ê\ado \|_{H^1_0(\Omega)}^2 + \|Ê\bdo  \|_{H^1_0(\Omega)}^2\big)^{\frac{1}{2}}\,  \big(\|Ê\ta \|_{H^1_0(\Omega)}^2 + \|Ê\tb  \|_{H^1_0(\Omega)}^2\big)^{\frac{1}{2}}.
\end{equation}
Moreover, Proposition \ref{prop:coer} entails
\begin{equation}\label{coer:opA}
\frac{1}{C}\big( \|\ado\|_{H^1_0(\Omega)}^2 + \|\bdo\|_{H^1_0(\Omega)}^2\big) \leq \|\bLcd(\ado,\bdo)\|_{L^2(\Omega)}^2.
\end{equation}
With the boundedness \eref{bound:opA} and coercivity \eref{coer:opA} now verified, the existence of a unique solution to the variational problem \eref{weak:form} follows directly from the Lax-Milgram theorem which ensures that
\[
\big( \|\ado\|_{H^1_0(\Omega)}^2 + \|\bdo\|_{H^1_0(\Omega)}^2 \big)^{\frac{1}{2}} \leq C \| \bfd-\bLcd (\alpha^\ell,\beta^\ell) \|_{L^2(\Omega)}.
\]


\section{Stability of reconstruction formula with noisy data}\label{sec:stab:noise}

\subsection{Solution stability w.r.t. noisy operator $\bLcd$}

Looking forward to the proposition of a regularization scheme in Section \ref{sec:num:reg} we now derive a stability result in the presence of noise. Considering Assumptions \ref{regularity_alpha} and Eqn. \eref{lifting_notation}, let $(\alpha^o,\beta^o)$ and $(\ado,\bdo)$ denote the solutions in $H^1_0(\Omega)^2$ to Eqn. \eref{eq:normal} that are respectively associated with exact measurements and noisy data satisfying Assumptions \ref{hyp:LB}. Then one has
\[\bLc (\alpha^o,\beta^o) = \bff-\bLc (\alpha^\ell,\beta^\ell)\quad\text{and}\quad\bLcd (\ado,\bdo) = \bfd-\bLcd (\alpha^\ell,\beta^\ell)+\bfhd
\]
where the subscript is omitted in noise-free quantities $(\alpha,\beta)\equiv(\alpha_0,\beta_0)$, $\bLc\equiv\bLc_0$ and $\bff\equiv\bff_0$, see Section \ref{sec:noisy:data}. Accounting for the fact that $\bfd$ no longer belongs to the range of the operator $\bLcd$, the additional term $\bfhd$ satisfies $\bLcsd\bfhd=\bzero.$ Defining the operator
\[
    \bEcd(\ta,\tb) := \left(\mathbf{M}-\mathbf{M}_\delta\right)\left[\!\begin{array}{c} \ta \\ \tb \end{array}\!\right],
\]
then the previous equations entail
\[
\bLcsd\bLcd ( \ado-\alpha^o,\bdo-\beta^o )=\bLcsd\big(\bfd-\bff+\bEcd(\alpha,\beta)\big)
\]
where $\ado-\alpha^o=\bdo-\beta^o=0$ on the boundary $\p\Omega$. Therefore, upon taking the $L^2(\Omega)$-inner product in the above equation with the vector $( \ado-\alpha^o,\bdo-\beta^o )$, one obtains
 \[
 \| \bLcd( \ado-\alpha^o,\bdo-\beta^o)  \|^2_{L^2(\Omega)}  =  \big\langle \bfd-\bff+\bEcd(\alpha,\beta) ,\, \bLcd( \ado-\alpha^o,\bdo-\beta^o) \big\rangle_{L^2(\Omega)},
 \]
so that
\[
 \| \bLcd( \ado-\alpha^o,\bdo-\beta^o)  \|_{L^2(\Omega)}  \leq  \| \bfd-\bff+\bEcd(\alpha,\beta)\|_{L^2(\Omega)}.
 \]
Assumptions \ref{regularity_alpha} together with Proposition \ref{prop:coer} yield the fundamental result of this section:
\begin{proposition}\label{ineq:stab}
For all $\delta\geq0$, if Assumptions \ref{hyp:LB} hold and $\bfd\in L^2(\Omega)^{2d}$ then there exists $C>0$ which does not depend on $\delta$ such that
\[
  \|\ado-\alpha^o\|_{H^1_0(\Omega)}  + \|\bdo-\beta^o\|_{H^1_0(\Omega)}  \leq C \big[ \| \bfd-\bff \|_{L^2(\Omega)} + \| \bEcd \| \,\big].
\]
where $\|\bEcd \|^2:=\sum_{\bv = \ba,\bb,\bc,\bd} \|\bv-\bv_\delta \|_{L^2(\Omega)}^2$.
\end{proposition}

\subsection{Solution stability w.r.t. noisy displacement measurements}

We now clarify the meaning of \eref{eq:approx} and deduce a stability result for the reconstructed parameters in terms of the approximation quality associated with the noisy data.

\begin{theorem}\label{theo_conv} If assumptions \ref{mega_hyp} and \ref{hyp:LB} hold and there exists an increasing function $\eta(\delta)$  such that, for all $\delta>0$, the functions $  (\buid, \buiid)  $ and $  ( \beps_{1, \delta}, \beps_{2, \delta}, \bHc_{1, \delta}, \bHc_{2, \delta} ) $ satisfy
\begin{equation}\label{rel:estimate_stab}
\sum_{n=1}^2\bigg\{ \|\bu_{n, \delta}  -  \bu_{n}\|_{L^\infty(\Omega)} +  \| \beps_{n, \delta}  - \beps_n\|_{L^\infty(\Omega)} +  \sum_{k=1}^d  \| \bHc_{n, \delta}^k -  \hess (\bu_{n}\cdot\be_k)\|_{L^2(\Omega)} \bigg\}\leq \eta(\delta),		
\end{equation}
then, for $ \delta $ small enough, the solution to \eref{eq:normal} with $(\bLcd, \bfd)$ defined in Section \ref{sec:noisy:data} satisfy
\[
    \|\ado-\alpha^o\|_{H^1_0(\Omega)} + \|\bdo-\beta^o\|_{H^1_0(\Omega)} \leq C \, \eta(\delta),
    \]
where $C$ is a positive scalar which does not depend on $\delta$.
\end{theorem}

\begin{proof} 
	The proof uses standard estimates as well as Assumptions \ref{mega_hyp}. A direct application of Proposition \ref{ineq:stab} is used to conclude. 
\end{proof}

\begin{remark}
	As the proof of Theorem \ref{theo_conv} relies on bounding the term $\|\bEcd \|$, it requires to estimate either the strain tensor $\beps_{n, \delta}$ in the $L^\infty$-norm and the hessian tensor $\bHc_{n, \delta}^k$ in the $L^2$-norm, or conversely. The choice adopted in \eref{rel:estimate_stab} is actually the more practical. 
\end{remark}

\begin{remark}
Theorem \ref{theo_conv} is consistent with the derivation of Eqn. \eref{eq:normal} in that no error has been considered in the knowledge of the boundary values of $ \alpha$ and $\beta $ on $\partial \Omega$.
\end{remark}

\section{Numerical differentiation scheme }\label{sec:num:reg}

We now address the question of  constructing explicitly the quantities $\{ \beps_{1, \delta}, \beps_{2, \delta}\}$ and $\{\bHc_{1, \delta}, \bHc_{2, \delta} \}$ from noisy measurements ${\buid, \buiid}$. In \cite{Kohn:1988} the discrepancy between noisy and exact measurements is assumed to be small in $H^1$-norm. Yet, in practice the functions $  \bu_{n,\delta} $ may not have bounded gradients or second-order derivatives. Therefore, we aim at defining a differentiation operator
\[
 \bDc :   (\buid, \buiid) \; \in \; \Big[ L^{\infty}(\Omega)^d\Big]^2 \rightarrow (\beps_{1, \delta}, \beps_{2, \delta}, \bHc_{1, \delta}, \bHc_{2, \delta}) \; \in \; \Big[L^{\infty}(\Omega)^{d\times d} \Big]^2 \times \Big[L^2(\Omega)^{d\times d\times d} \Big]^2   ,
\]
such that, for $ \delta $ small enough, Eqn. \eref{rel:estimate_stab} holds with a function $\eta(\delta)$ to be determined. 
Therefore, we are required to regularize the measurements to define the operator $\bDc$. While various methods have been discussed in the literature, see e.g. \cite{Engl}, we discuss hereinafter a regularization by $L^2$-projections onto coarse finite element spaces and show in which sense this approach is optimal.

\subsection{Approach overview} 

Consider a partition of the domain $\Omega \subset \mathbb{R}^d$ using a number of finite elements $K_e$ that are characterized by a mesh size $h$ with $0<h<1$ and such that
\begin{equation}\label{eq:hypo_maillage}
\overline{\Omega} = \bigcup_e \overline{K_e}\quad\text{with}\quad  K_e \cap K_{e'} = \emptyset \quad \text{if } e \neq e' \quad \text{and}\quad \int_{K_e} 1 \, \td \bx \leq C h^d,
\end{equation}
with $C$  independent of $h$. This partition is associated with the finite element spaces
\[
L_{h}^r = \{ v_h \in L^2(\Omega) \quad | \quad (v_h)|_{K_e} \in  Q^r_e  \quad \forall \; K_e \subset \Omega \; \}
\]
with $Q^r_e$ a local finite element space on $K_e$ that includes at least the polynomials of order $r$. Let $V_{h}\subset H^1_0$ denote a finite element space approaching $H^1_0(\Omega)$ in the limit $h\to0$.\\

The approach considered is as follows: Noisy displacement fields measurements $\bu_{n,\delta}$ are assumed to be available on a fine grid. They are interpreted as functions whose components belong to a low-order finite element space, such as $L^0_{h_0}$, on a fine mesh. Note that this choice is arbitrary and $L^1_{h_0}$ or a standard conforming finite element space can be used for interpolating displacement data. Next, we define regularization and differentiation schemes to construct the strain and hessian tensors $\beps_{n,\delta}$, $\bHc_{n,\delta}$ of these quantities. They are associated to coarser meshes but higher order finite elements, i.e. to functional spaces $L^r_{h_1}$ and $L^r_{h_2}$ respectively, with $h_0\ll h_1\leq h_2$. Reconstruction of moduli $(\alpha^o_\delta,\beta^o_\delta)$ is finally achieved upon solving \eref{eq:normal} in $V_{h_0}$ on the fine discretization.

\begin{table}[h]
\caption{Summary of the approach where $r\geq2$ and $h_2\geq h_1\gg h_0$.}
\begin{center}\vspace{-4mm}
\begin{tabular}{c|c c c c}
& Data & Gradient & Hessian & Reconstruction \\
\hline\hline \\[-3mm]
 Quantities & $\bu_{n,\delta}$ & $\beps_{n,\delta}$ & $\bHc_{n,\delta}$ & $(\alpha^o_\delta,\beta^o_\delta)$  \\[1mm] 
Fct. spaces & $L^0_{h_0}$ & $L^r_{h_1}$ & $L^r_{h_2}$ & $V_{h_0}$\\
\end{tabular}
\end{center}\label{Table_space}
\end{table}%

\noindent Note that in the ensuing analysis, we do not account for the numerical errors associated with (i) the numerical quadratures required for the computation of $\beps_{n,\delta}$ and $\bHc_{n,\delta}$, and (ii) the computation of $(\alpha^o_\delta,\beta^o_\delta)$ using the discretized version of the variational formulation discussed above.

\subsection{The $L^2$-projection operator}

Let $\bPc_h^r$ denote the scalar orthogonal projection operator on the finite element space of order $r$ considered
\[
\bPc_h^r :  u \in  L^2(\Omega) \longrightarrow  u_{h,r} \in L_{h}^r
\]
where $ u_{h,r} $ is defined as the unique function of $L_{h}^r$ such that
\[
 ( u_{h,r} , v_{h,r})_{L^2(\Omega)} = ( u, v_{h,r})_{L^2(\Omega)} \quad \forall \; v_{h,r} \in L_{h}^r.
\]
As the spaces $L_{h}^r$ are not conforming this equation is equivalent to have, for all $K_e \subset \Omega$, 
\begin{equation}\label{eq:projection_local}
( u_{h,r} , v_{h,r})_{L^2(K_e)} = ( u, v_{h,r})_{L^2(K_e)} \quad \forall \; v_{h,r} \in Q^r_e,	
\end{equation}
which defines the local orthogonal projection operator $ \bPc^r_{h,e}$ that is associated with element $ K_e $ and satisfies  $ \bPc^r_{h,e}(u{|_{K_e}}) = u_{h,r}{|_{K_e}} $. Standard projection properties are satisfied locally and globally, i.e.
\[\begin{aligned}
    \bPc^r_{h,e}({u_{h,r}}{|_{K_e}}) = {u_{h,r}}{|_{K_e}} &\quad\text{and} \quad  \bPc_{h}^r(u_{h,r}) = u_{h,r} \quad \forall \; u_{h,r} \in L_{h}^r \\
    \| \bPc^r_{h,e}(u{|_{K_e}}) \|_{L^2(K_e)} \leq C \, \| u \|_{L^2(K_e)} &\quad\text{and} \quad \| \bPc_{h}^r(u) \|_{L^2(\Omega)} \leq C \,  \| u \|_{L^2(\Omega)}  \quad \forall u \in L^2(\Omega).    
\end{aligned}\]
We now state an approximation result, which can be proven using standard results of the theory of finite element method \cite{Brenner:2002hf,Heimsund:2003va}. This result quantifies how the local $L^2$-projection onto high-order finite element spaces yields a smooth local approximation of a given function $u$ from a $L^2$ approximation of this function. In these inequalities, $ r $ denote the order of the featured finite element space while the positive integers $\ell$ and $ m $ stand for the Sobolev indices respectively associated with the {\it a priori} regularity of the function $u$ considered and the required regularity of the approximation.
\begin{proposition}\label{prop:reg}
For any $u \in  H^{\ell}(K_e) $, $  v \in  L^{2}(K_e) $ and element $K_e \subset \Omega $ the following inequalities hold 
\begin{itemize}
    \item[--]
	If $ m < \ell\leq r+1$, we have
	\begin{equation}\label{eq:ineq_1}
	    \|  u -  \; \bPc^r_{h,e} ( v) \|_{H^m(K_e)}  \leq   C \, h^{-m} \|u - v \|_{L^2(K_e)} + C \, h^{\ell-m} \, | u |_{H^{\ell}(K_e)}.
	\end{equation}
    \item[--]
	If  $ m+d/2 < \ell\leq r+1$ then
	\begin{equation}\label{eq:ineq_2}
	    \|  \; u -  \; \bPc^r_{h,e} ( v) \|_{W^{m,\infty}(K_e)} \leq    C \, h^{-m-d/2} \, \|u - v \|_{L^2(K_e)} + C \, h^{\ell-m-d/2} \, | u |_{H^{\ell}(K_e)}.
	\end{equation}
\end{itemize}
\end{proposition}

\noindent The interpretation of Proposition \ref{prop:reg} is as follows: Within the framework of finite element spaces, one can approximate locally in element $K_e$, the derivatives up to order $m$ of a given function $u \in  H^{\ell}(K_e) $ using a function $v \in  L^{2}(K_e) $ that approximates $u$ in the $L^2$-norm only. Yet, the approximation quality is penalized by the detrimental term $h^{-m}$, with $h$ being the mesh size. Therefore, to obtain an optimal approximation in the sense of Theorem \ref{theo_conv}, the parameter $h$ has to be chosen in such a way that, roughly speaking, $\delta h^{-m}\ll1$. This issue is discussed in the next section.

\subsection{Construction of noisy operator ${\bLcd}$ by regularization}\label{sec_reg}

Consider noisy displacement measurements $\buid, \buiid \in L^\infty(\Omega)^d$ satisfying for all $\delta > 0$
\begin{equation}\label{eq:measure}
	\|\bui -  \buid \|_{L^\infty(\Omega)} + \|\buii -  \buiid \|_{L^\infty(\Omega)} \leq C \, \delta.
\end{equation}
The differentiation operator $\bDc$ is now given by
\begin{equation}\label{def:op:diff}
 \bDc(\buid, \buiid) =  (\beps_{1, \delta}, \beps_{2, \delta}, \bHc_{1, \delta}, \bHc_{2, \delta})
\end{equation}
where the strain and hessian tensors are defined element-wise in all elements $K_e \subset \Omega$ by
\[
\beps_{n, \delta}{|_{K_e}} = \nabs \bigg(\sum_{k=1}^d\bPc_{h_1,e}^r (\bu_{n,\delta}{|_{K_e}}\cdot\be_k)\be_k \bigg), \quad  \bHc_{n, \delta}{|_{K_e}} =  \sum_{k=1}^d\hess\big(  \bPc_{h_2,e}^r ( \bu_{n,\delta} {|_{K_e}}\cdot\be_k) \big)\otimes\be_k.
\]
For the sake of generality, two different meshes of sizes $ h_1 $ and $ h_2 $ are used to project the data $(\buid, \buiid)$ to compute first and second-order derivatives respectively.

For simplicity, we assume that the function $\eta(\delta)$ in \eref{rel:estimate_stab} is given by a power law. Now, the aim is to tune the parameters $ h_1 $ and $ h_2 $ as functions of $ \delta $ to obtain the best possible estimate in \eref{rel:estimate_stab}. For the sake of argument, let $\ell \geq \ell_0 \geq 2$ such that $\bu_{n},  \bu_{n,\delta} \; \in \; H^{\ell}(\Omega) \cap W^{\ell-1,\infty}(\Omega)$ for all $\delta>0$ and $n=1,2$. In this case one can show using classic interpolation results, see \cite[Theorem 4.17]{Adams:1975}, that a direct differentiation of the measurements $\buid, \buiid $ yields the convergence rate $\eta(\delta) = \delta^{\frac{\ell-2}{\ell}}$. Therefore this is the optimal bound that constrains the final estimate on the reconstruction.

\begin{theorem}\label{th:final} Let Assumptions \ref{mega_hyp} be satisfied and assume that the solutions to \eref{eq:sol_general} satisfy $\bu_1, \bu_2 \in H^\ell(\Omega)^d$ where $\ell$ is an integer such that $ 3 \leq \ell \leq r+1$, while the noisy measurements satisfy \eref{eq:measure}. For $\delta$ sufficiently small, when choosing 
	\[
	h_1 =  \delta^{\frac{1}{\ell-d/2}} \quad \mbox{ and } \quad h_2 = \delta^{\frac{1}{\ell}}
	\]	
and constructing the operator $\bLcd$ using $  \bDc $ defined by \eref{def:op:diff}, then the reconstructed parameters $ (\ado, \bdo) $ obtained solving \eref{eq:normal} satisfy
	\begin{equation}\label{eq:bound}
		\|\ado-\alpha^o\|_{H^1_0(\Omega)} + \|\bdo-\beta^o\|_{H^1_0(\Omega)} \leq C \, \delta^{\frac{\ell-2}{\ell}},
	\end{equation}
		where the constant $C>0$ does not depend on $\delta$.
\end{theorem}

\begin{proof}
	From Definition \eref{def:hess} and Property \eref{eq:ineq_1} we have
	\[\begin{aligned}
 \|\bHc^k_{n, \delta} \,- \, \hess (\bu_{n}\cdot\be_k)\|_{L^2(K_e)} & \leq C \| \bPc_{h_2,e}^r (\bu_{n,\delta}) - \bu_{n}\|_{H^2(K_e)} \\[4pt]
	& \leq  C \, {h_2}^{-2} \|\bu_{n} - \bu_{n,\delta}  \|_{L^2(K_e)} + C \, {h_2}^{\ell-2} \, | \bu_{n} |_{H^{\ell}(K_e)} . 
	\end{aligned}\]
	Squaring the above, summing over elements $K_e$ and taking the square root, entail
	\[
 \|\bHc^k_{n, \delta} \,- \, \hess (\bu_{n}\cdot\be_k)\|_{L^2(\Omega)} \leq C \, {h_2}^{-2} \|\bu_{n} - \bu_{n,\delta}  \|_{L^2(\Omega)} + C \, {h_2}^{\ell-2} \, \big( \sum_{K_e} | \bu_n |_{H^{\ell}(K_e)}^2 \big)^{1/2} .
	\]
	As $\bu_1, \bu_2 \in H^\ell(\Omega)^d$, using Eqn. \eref{eq:measure} and the boundedness of $ \Omega $ finally yield
	\begin{equation}\label{eq:first_part_theo_conv}
		\|\bHc^k_{n, \delta} \,- \, \hess (\bu_{n}\cdot\be_k)\|_{L^2(\Omega)} \leq C \, \bigg(\frac{\delta}{{h_2}^{2}} + {h_2}^{\ell-2}  \bigg).
		\end{equation}
	Moreover, owing to Property \eref{eq:ineq_2}, one has
	\[\begin{aligned}
		\|\beps_{n, \delta} \, - \,\beps_n \|_{L^\infty(\Omega)}  & \leq  \sup_{K_e} \|\bPc_{h_1}^r (\bu_{n,\delta}) - \bu_{n}\|_{W^{1,\infty}(K_e)} \\ & \leq C \sup_{K_e} \big(  \, h_1^{-1-d/2} \, \|\bu_{n} - \bu_{n,\delta} \|_{L^2(K_e)}  +  \, h_1^{\ell-1-d/2} \,  | \bu_{n} |_{H^{\ell}(K_e)}  \big).
	\end{aligned}\]
	Using Eqn. \eref{eq:measure}, the first term can be bounded uniformly with respect to $K_e \subset \Omega$ as
	\[
	 h_1^{-1-d/2} \, \|\bu_{n} - \bu_{n,\delta}\|_{L^2(K_e)} \leq C h_1^{-1} \, \|\bu_{n} - \bu_{n,\delta}\|_{L^\infty(K_e)} \leq C \, h_1^{-1} \delta,
		\]
	since the mesh is assumed to be uniform for all $h_1$ owing to \eref{eq:hypo_maillage}. The above result and the inequality $ \sup_{K_e}  | \bu_{n}  |_{H^{\ell}(K_e)}  \leq C | \bu_{n} |_{H^{\ell}(\Omega)}  $ yield 
	\begin{equation}\label{eq:second_part_theo_conv}
\|\beps_{n, \delta} \, - \,\beps_n \|_{L^\infty(\Omega)}  \leq C \bigg( \frac{\delta}{ h_1} + h_1^{\ell-1-d/2} \bigg).
	\end{equation}
	Combining \eref{eq:first_part_theo_conv} and \eref{eq:second_part_theo_conv} finally entails
	\[
\sum_{n=1}^2\!\bigg\{\! \|\beps_{n, \delta}  - \beps_n \|_{L^\infty(\Omega)} +  \sum_{k=1}^d  \|\bHc^k_{n, \delta} -  \hess (\bu_{n}\cdot\be_k)\|_{L^2(\Omega)} \!\bigg\} \leq C  \bigg( \frac{\delta}{h_1} \,+\, h_1^{\ell-1-d/2} \,+\, \frac{\delta}{h_2^{2}} \,+\, h_2^{\ell-2}  \bigg).
	\]
 	To conclude the proof, let $h_1 = \delta^{\mathrm{r}_1}$ and $h_2 = \delta^{\mathrm{r}_2}$ with the positive real parameters $\mathrm{r}_1$ and $\mathrm{r}_2$ chosen to maximize the quantities 
$
	 \min(\mathrm{r}_1(\ell-1-d/2),1-\mathrm{r}_1) $ and $ \min(\mathrm{r}_2(\ell-2),1-2\mathrm{r}_2) .
$
	Since $\ell \geq 3$, one can show that these maxima are achieved when $\mathrm{r}_1 = 1/(\ell-d/2)$ and $\mathrm{r}_2=1/\ell$ . Then Theorem \ref{theo_conv} is used to finally obtain Estimate \eref{eq:bound}.
\end{proof}

\begin{remark}
	When $ d=2 $, the choice $h_1 = h_2 = \delta^{1/\ell}$ is actually sufficient to obtain \eref{eq:bound}. However, in this case the featured constant $C$ would be suboptimal compared to this obtained using Theorem \ref{th:final}.
\end{remark}

\section{Numerical results}\label{sec:num:res}

A set of numerical results is presented in this section to assess the performances of the proposed approach. These numerical examples correspond to solving Eqn. \eref{eq:normal} in dimension $d=2$ with the noisy operators constructed using the method presented in Section \ref{sec_reg} with $h_1=h_2\equiv h$. Given a distribution of smooth enough constitutive parameters $(\alpha, \beta)$, the solutions $(\bu_1,\bu_2)$ are computed on a sufficiently fine mesh from Eqn. \eref{REFequ} augmented with non-homogeneous Dirichlet boundary conditions
\begin{equation}
 \bu_n=\bg_n\; \;  \text{ on } \partial \Omega, \qquad n=1,2.
\label{REFequBC}
\end{equation}
These solutions are then polluted by a noise, which, for the sake of reproducibility, is parameterized and therefore deterministic, as an increasing function in $L^\infty(\Omega)$-norm of a parameter $\delta$. In what follows, we assess the behavior of the proposed numerical algorithm with respect to the parameters $h$ and $\delta$. Let $\Omega = [0,1]^2$ and the parameters $\alpha(\bx)$ and $\beta(\bx)$ by defined as 
\begin{equation}\label{eq:formulae_alpha_beta}
	\alpha(\bx) = 	\alpha_0 + \sum_{i=1}^{N_\alpha} 	\alpha_i \, c( |\bx -\,\bx_i^\alpha| \, ; \, a_i^{-}, a_i^{+}), \quad 	\beta(\bx) = \beta_0 + \sum_{i=1}^{N_\beta} \beta_i \, c( |\bx - \bx_i^\beta|  \, ; \, b_i^{-}, b_i^{+}), 
\end{equation}
where the function $ c(r \, ; \, r^-,  r^+) $ equals $ 1 $ if $ r  < r^-$, $ 0 $ if  $ r > r^+ $ and
\[
	c(r \, ; \, r^- , r^+) = \!\big( 1 - {(r-r^-)}/{(r^+ - r^-)} \big)^{\!2} \, \big(1+2 {(r-r^-)}/{(r^+ - r^-)} \big) \quad  r^- \leq r \leq  r^+.
\]
The radial function $ c(|\bx| \, ; \, r^-,  r^+) $ is $ C^1(\Omega) $ while all its second order derivatives are $ L^\infty(\Omega) $. This property guarantees that the solutions $\bu_1$, $\bu_n$ to \eref{REFequ} and \eref{REFequBC} are at least in $H^3(\Omega)$ for smooth enough boundary conditions $\bg_n$. In turn, it ensures that Assumptions \ref{mega_hyp} are satisfied owing to the injection of $H^3(\Omega)$ into $W^{1,\infty}(\Omega)$. Therefore, the algorithm can at least be defined at the continuous level and Theorem \ref{th:final} grants us theoretical convergence of the algorithm as $\delta\to0$.

Numerical computations are performed using continuous fifth-order nodal finite elements on a square mesh based on Gauss-Lobatto points as described in \cite{Cohen:2002ta}. We use fifth-order quadrature formulae based on Gauss points for the local projection \eref{eq:projection_local} and Gauss-Lobatto points for the computation of the finite element matrices. The synthetic measurements $\bu_1$ and $\bu_2$ are computed using a conjugate gradient (CG) technique, on a reference fine grid of size $h_0=1/120$, i.e. $L^0_{h_0} = V_{h_0}$ and $\dim(V_{h_0}) = 601\times601$. Then, deterministic noise is added to the solution as
\[
	\bu_{n,\delta}(\bx)  = \bu_n(\bx)  + \delta \! \sum_{m = -M}^M \frac{|m|}{M} \,  \chi\left(\frac{|m|}{M} \, \frac{\bx}{ \sqrt{\delta} }\right),
\]
with $\chi(\bx) =  \cos(2 \pi \, \bx \cdot \be_1) \cos(2 \pi \, \bx \cdot \be_2) (\be_1 + \be_2)$ and $M=20$. At the discrete level, the noise is directly interpolated on the nodal functions generating $L^0_{h_0}$. Moreover, Assumption \eref{eq:measure} is satisfied by construction, yet $\| \bu_{n,\delta}(\bx)  -  \bu_n(\bx)\|_{H^2(\Omega)} =\mathcal{O}(1)$, which entails that a non-regularized direct data differentiation approach with $\dim V_{h_0} = +\infty$ should not converge as $\delta\to0$.

Finally, it is considered that the solution to the elliptic problem \eref{eq:normal} is computed sufficiently accurately with a CG method, i.e. the space $V_{h_0}$ in Table \ref{Table_space} is large enough and the relative stopping criterion of the CG method is good enough.

\paragraph{Simple static case.} The parameters $\alpha$ and $\beta$ are given by formulae \eref{eq:formulae_alpha_beta} with $ \alpha_0 = 22$, $ \beta_0 = 2$, $ N_\alpha =  N_\beta = 1$, $ \alpha_1 = \beta_1 = 18$, $ a_1^{-} = b_1^{-} = 0.1 $,  $  a_1^{+} = b_1^{+} = 0.2 $, $\bx_1^\alpha = \bx_1^\beta = [1/2, 1/2]\trsp$, see Fig. \ref{fig:beta:ref}. The static case corresponds to $\omega_1 = \omega_2 = 0$ in \eref{eq:sol_general}. As discussed in Section \ref{sec:invert:cond}, the invertibility condition in Assumptions \ref{mega_hyp} can be satisfied by choosing appropriate Dirichlet boundary conditions $\bg_1$ and $\bg_2$ in \eref{REFequBC}, such as
\begin{equation}\label{BC_res_num_1}
	\bg_1(\bx) = 1+( \bx \cdot \be_2 ) \, \be_1+( \bx \cdot \be_1 ) \, \be_2, \quad \bg_2(\bx) = 1+( \bx \cdot \be_1 ) \, \be_1+( \bx \cdot \be_2 ) \, \be_2.
\end{equation}

\begin{figure}[th]	
\centering
\subfloat[$\bu_1(\bx)\cdot \be_1$]{\includegraphics[height=0.122\textheight]{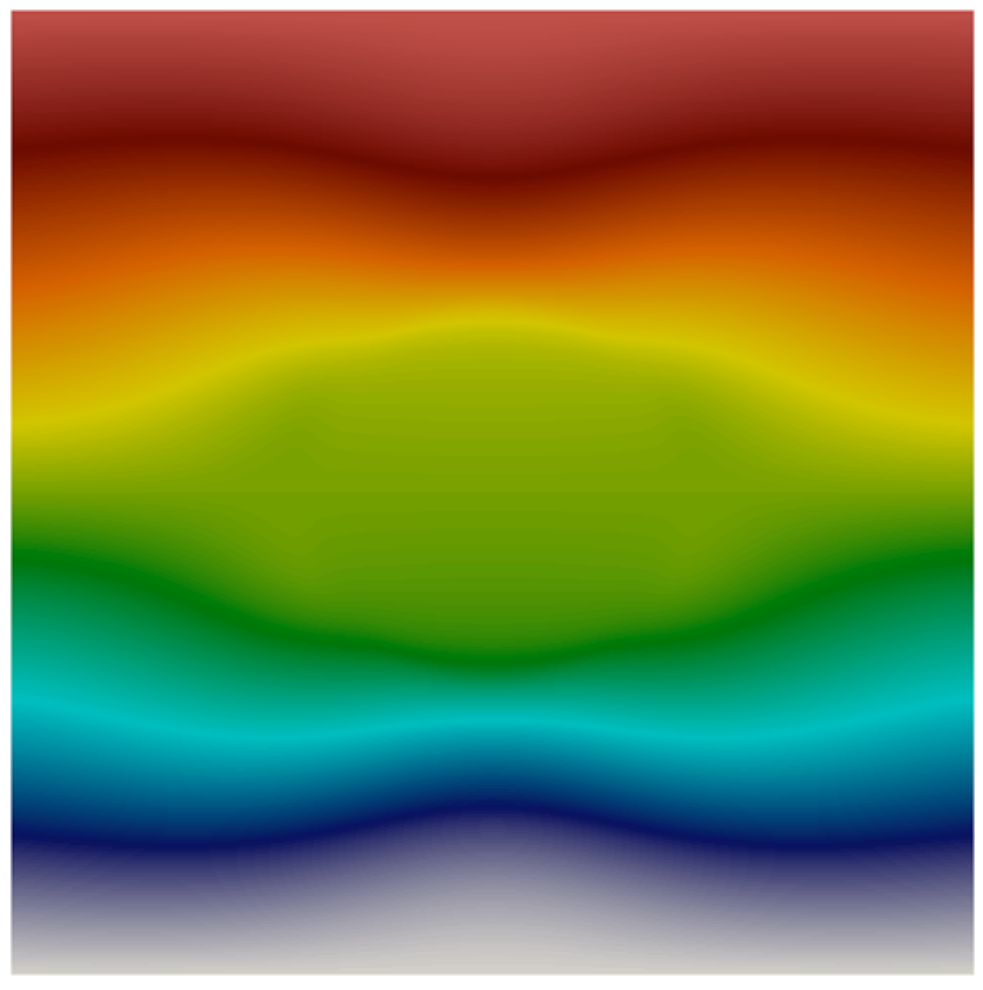}}
\subfloat[$\bu_1(\bx)\cdot \be_2$]{\includegraphics[height=0.122\textheight]{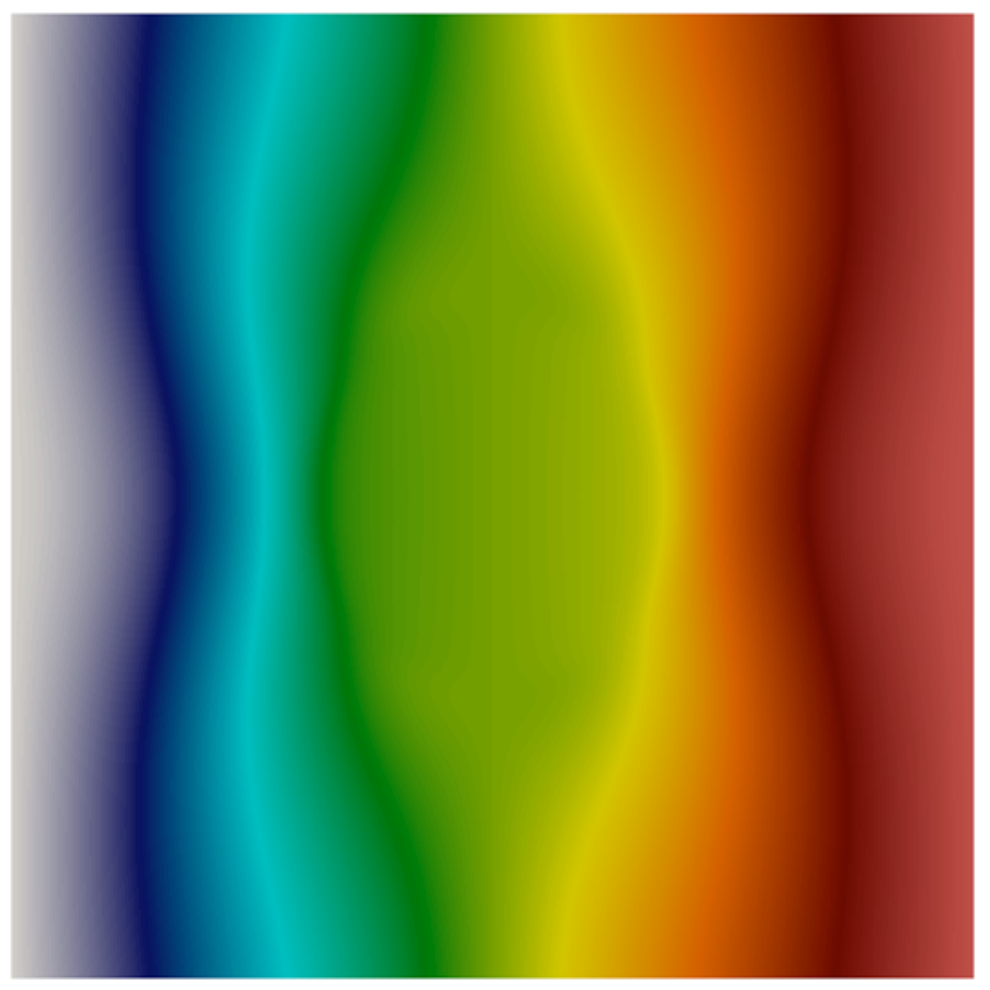}}
\subfloat[$\bu_2(\bx)\cdot \be_1$]{\includegraphics[height=0.122\textheight]{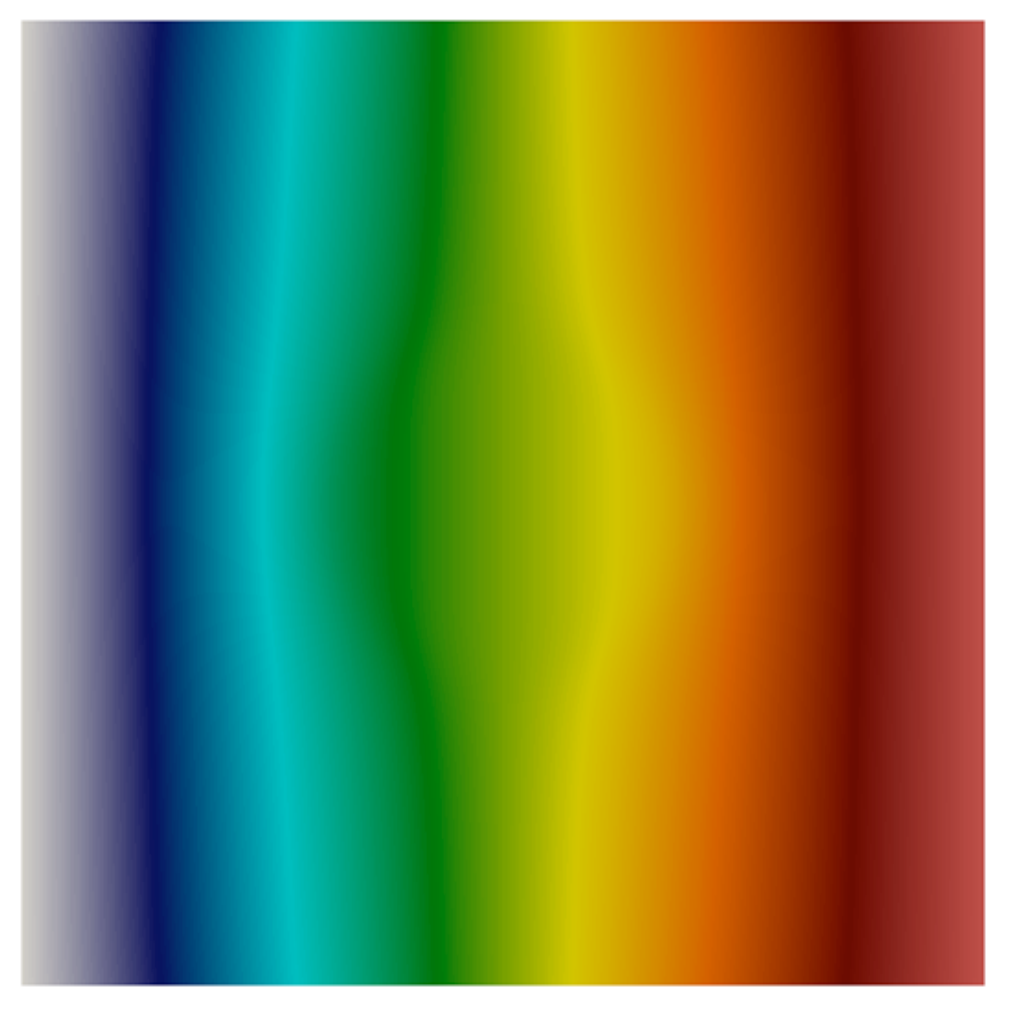}}
\subfloat[$\bu_2(\bx)\cdot \be_2$]{\includegraphics[height=0.122\textheight]{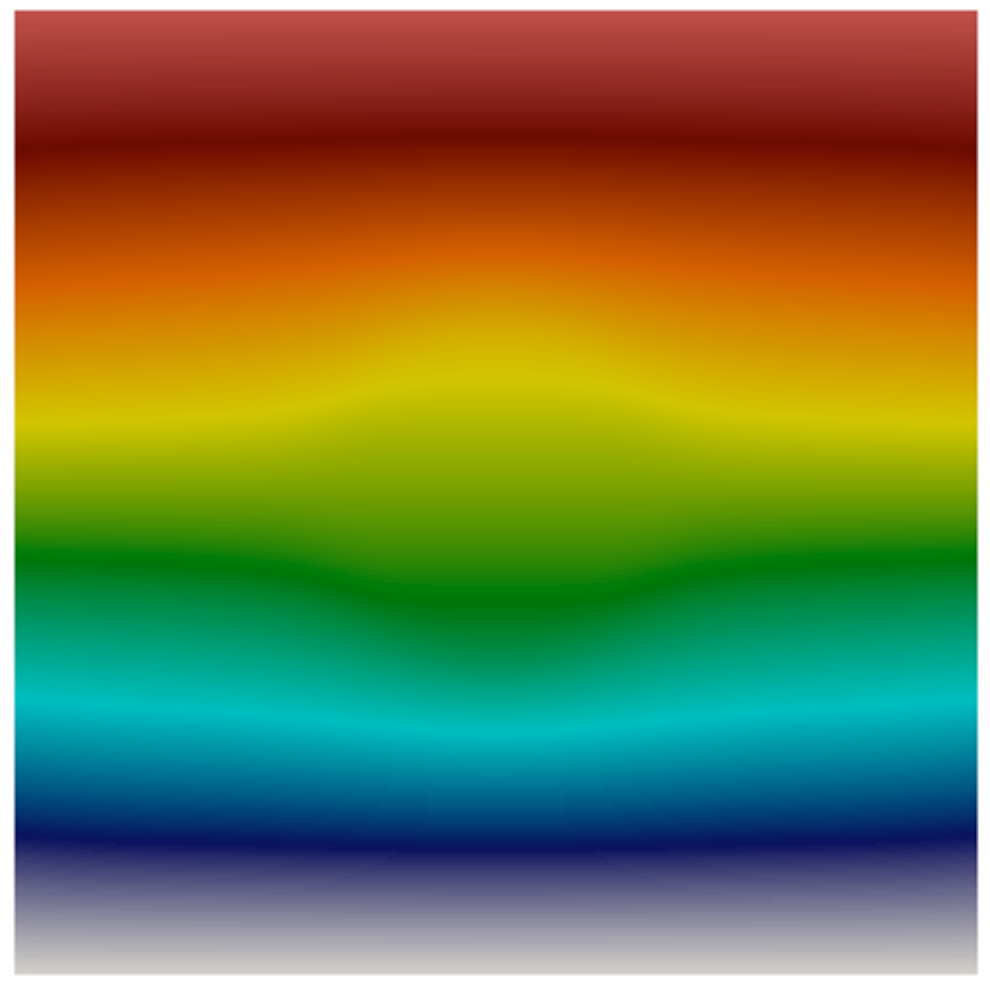}}
\subfloat{\includegraphics[height=0.122\textheight]{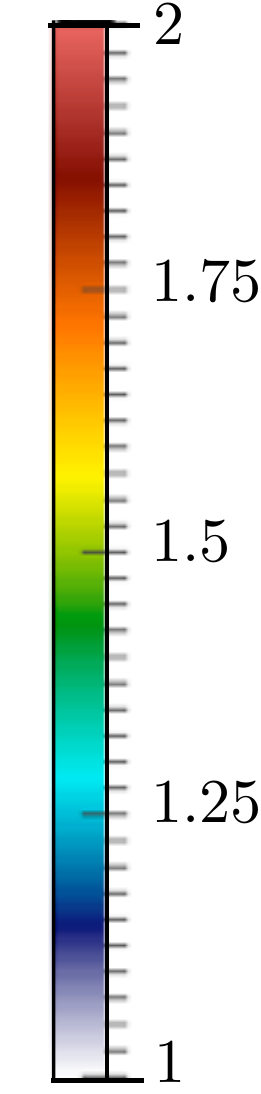}}
\subfloat[$\det\Escr(\bx)$]{\includegraphics[height=0.122\textheight]{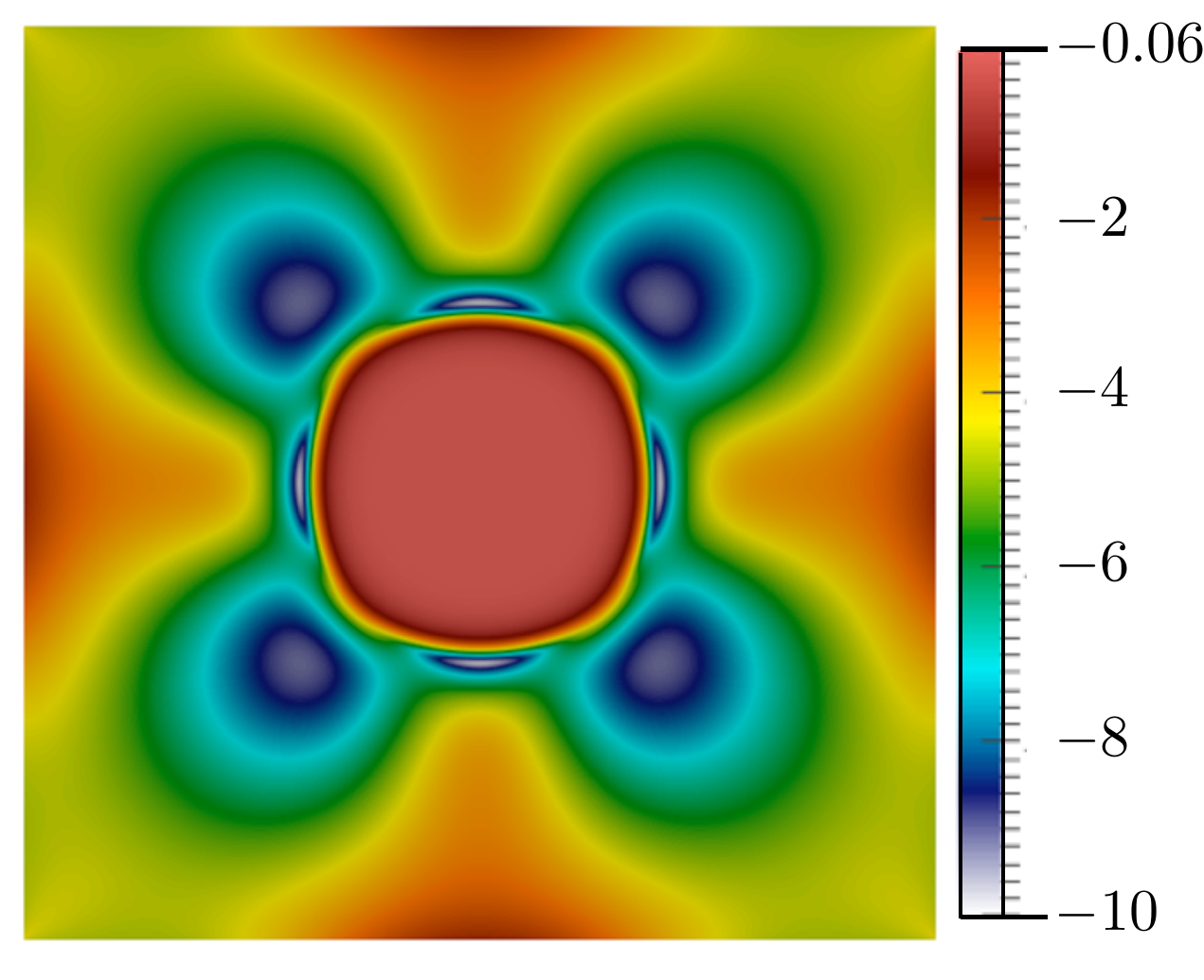}}
\caption{Computed elasticity solutions in the static case.}
\label{fig_solution}
\end{figure}
\noindent Computed elasticity solutions $\bu_1$ and $\bu_2$ are represented Figure \ref{fig_solution}, together with the term $\det\Escr$, defined by \eref{eq:def_E}, which does not vanish so that Assumptions \ref{mega_hyp} are satisfied.

\begin{figure}[t]	
\centering
\includegraphics[width=0.55\textwidth]{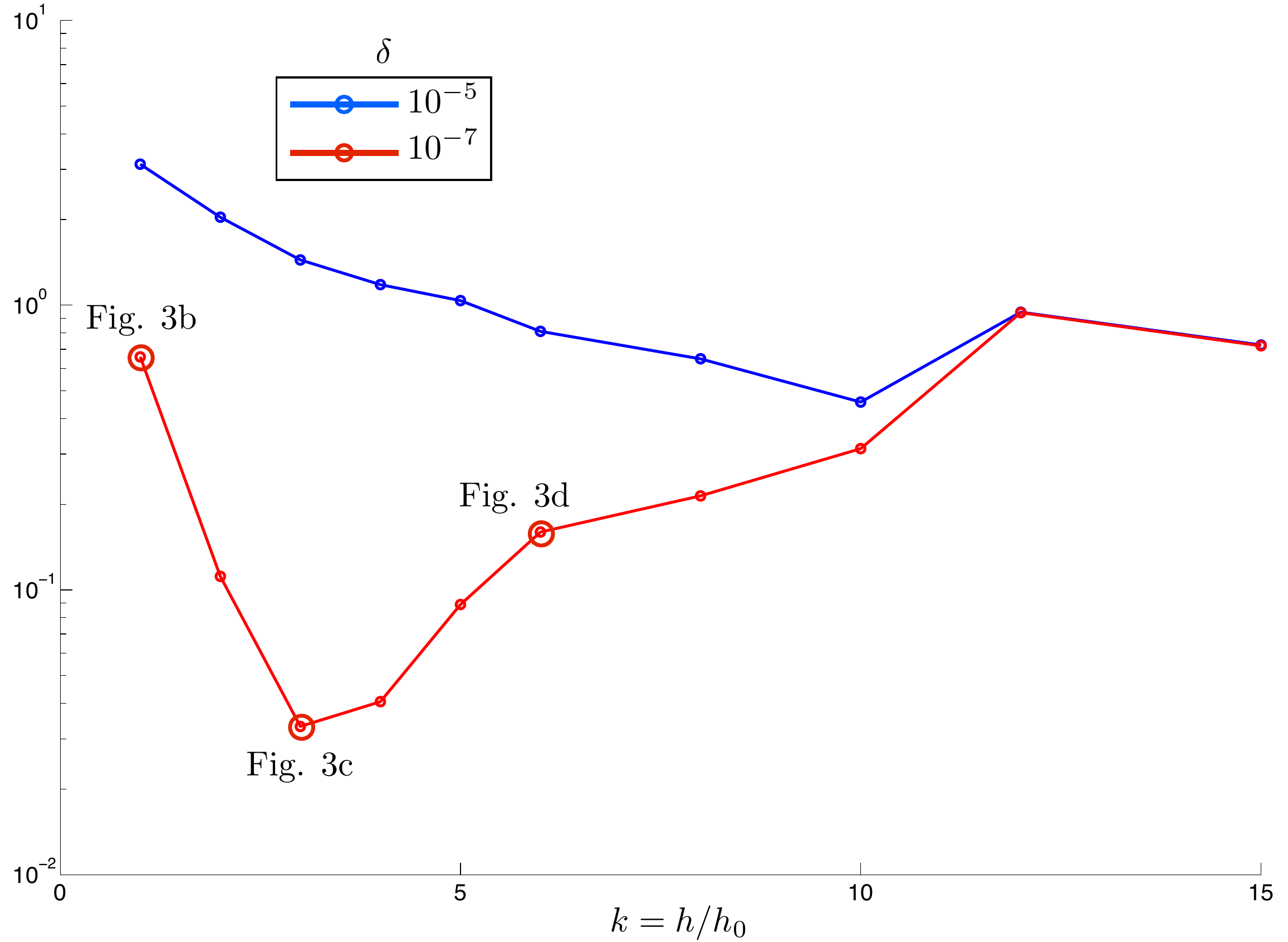}\vspace{-4mm}
\caption{Relative error in $H^1(\Omega)$-norm on reconstructed parameters $(\alpha,\beta)$ w.r.t. mesh size $h = k\, h_0$. The doubly-circled dots correspond to the reconstructions in Figure \ref{fig_h_varies}. }
\label{fig_h_convergence_plot}
\end{figure}

\textit{Convergence in $h$.} The reconstruction behavior is investigated by solving Eqn. \eref{eq:normal} for different mesh sizes $h$ in the definition of $\bLcd$ from Section \ref{sec_reg}. This parameter is varied as $h=k\,h_0$ using an integer $k$ and given the size $h_0$ of the reference mesh associated with the noisy data. Figure \ref{fig_h_convergence_plot} shows the relative $H^1(\Omega)$-error of the reconstruction for noise values $\delta = 10^{-5}$ and $\delta = 10^{-7}$. In the light of Eqns. \eref{eq:first_part_theo_conv} and \eref{eq:second_part_theo_conv}, Figure \ref{fig_h_varies} shows that the computation of the data derivatives is penalizing if the mesh employed is too fine, as negative powers of $h$ are involved. Alternatively, if the mesh is too coarse then the quality of the solutions approximation is too deteriorated, as they involves positive powers of $h$. Hence, there exists an optimal value $h$ corresponding, in each case, to the minimum of the associated curve in Figure \ref{fig_h_convergence_plot}.

\begin{figure}[t!]	
\centering
\subfloat[$\beta(\bx)$]{\includegraphics[height=0.137\textheight]{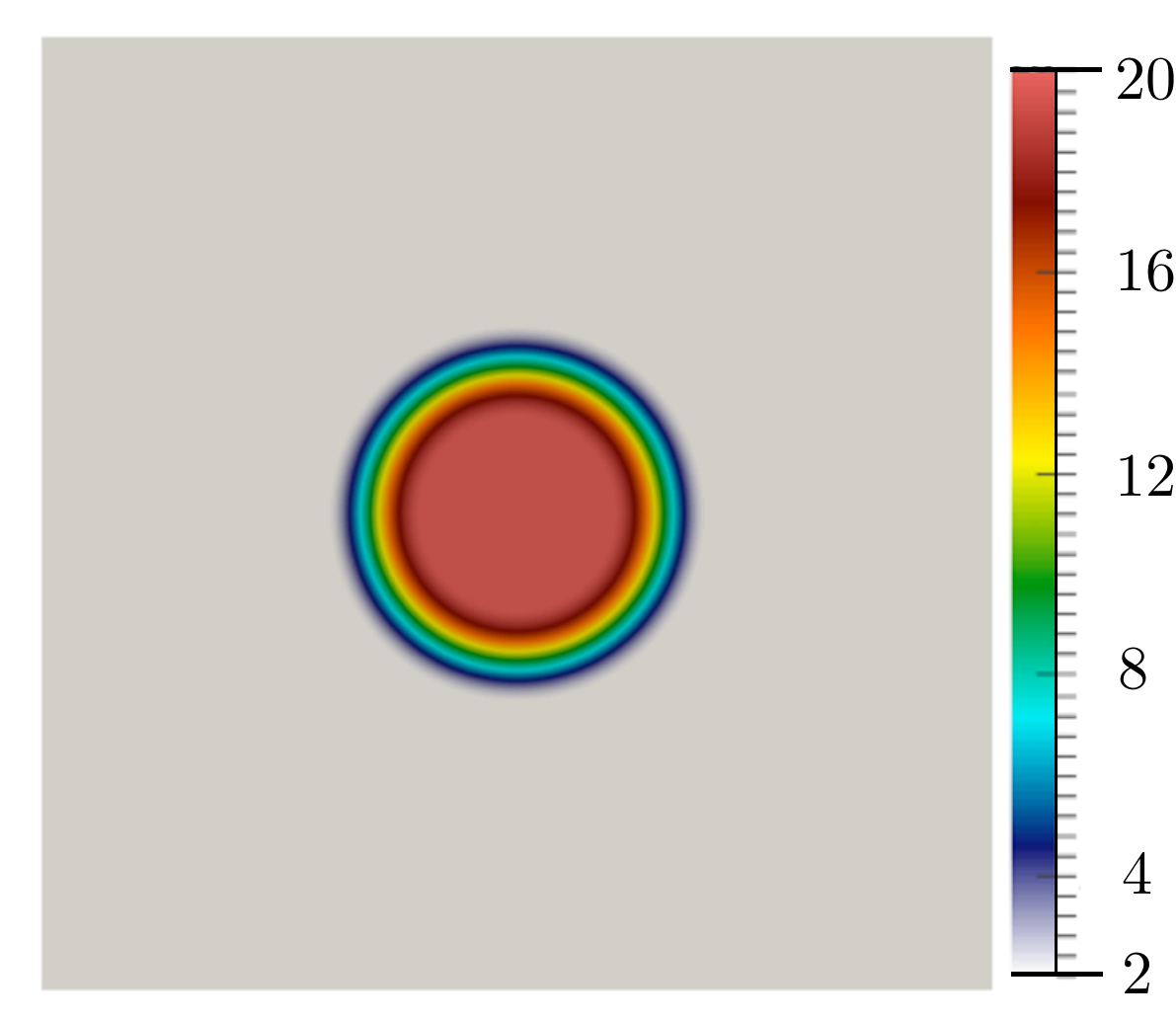} \label{fig:beta:ref}}
\subfloat[$k=1$]{\includegraphics[height=0.137\textheight]{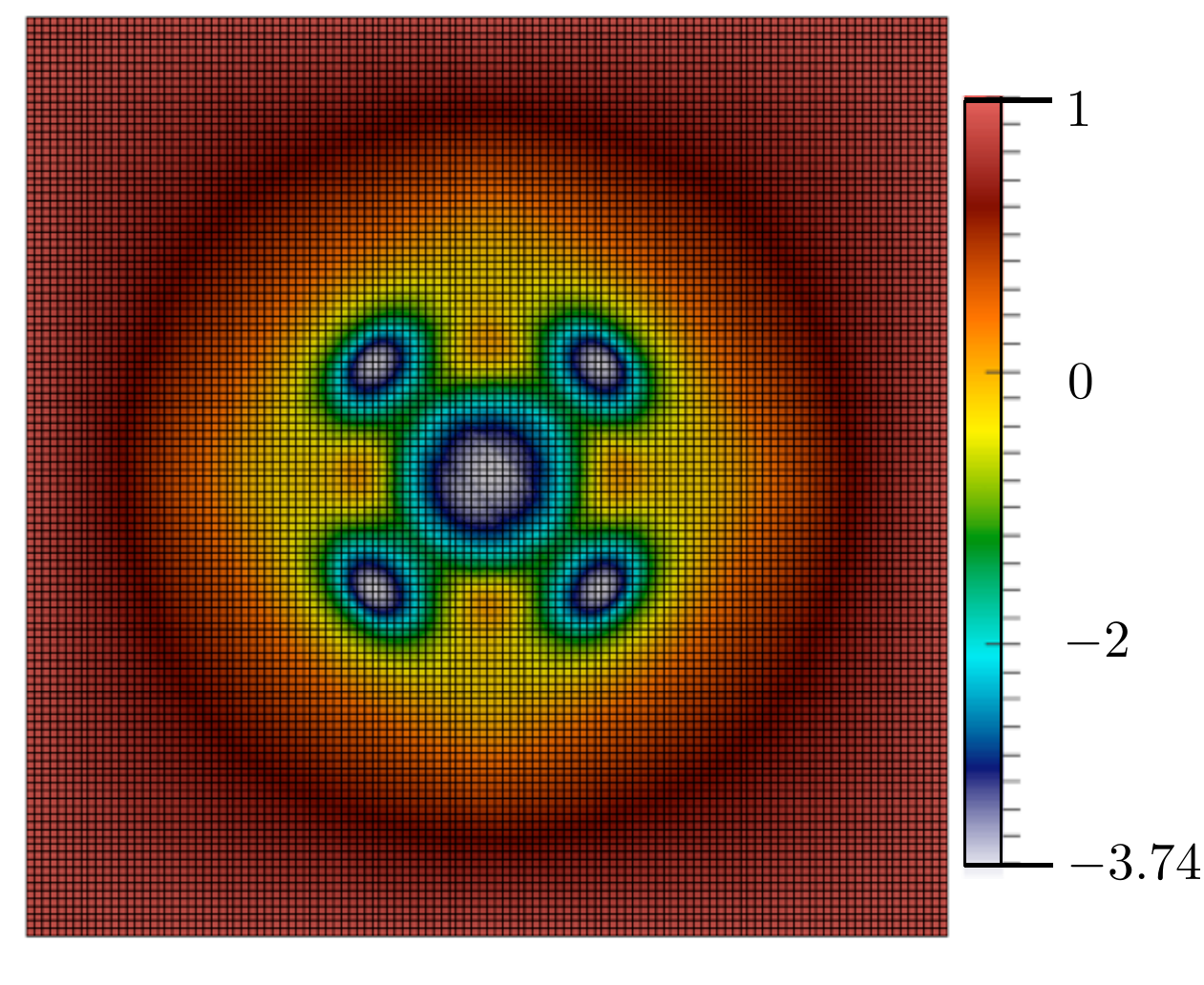}}\hspace{-1mm}
\subfloat[$k=3$]{\includegraphics[height=0.137\textheight]{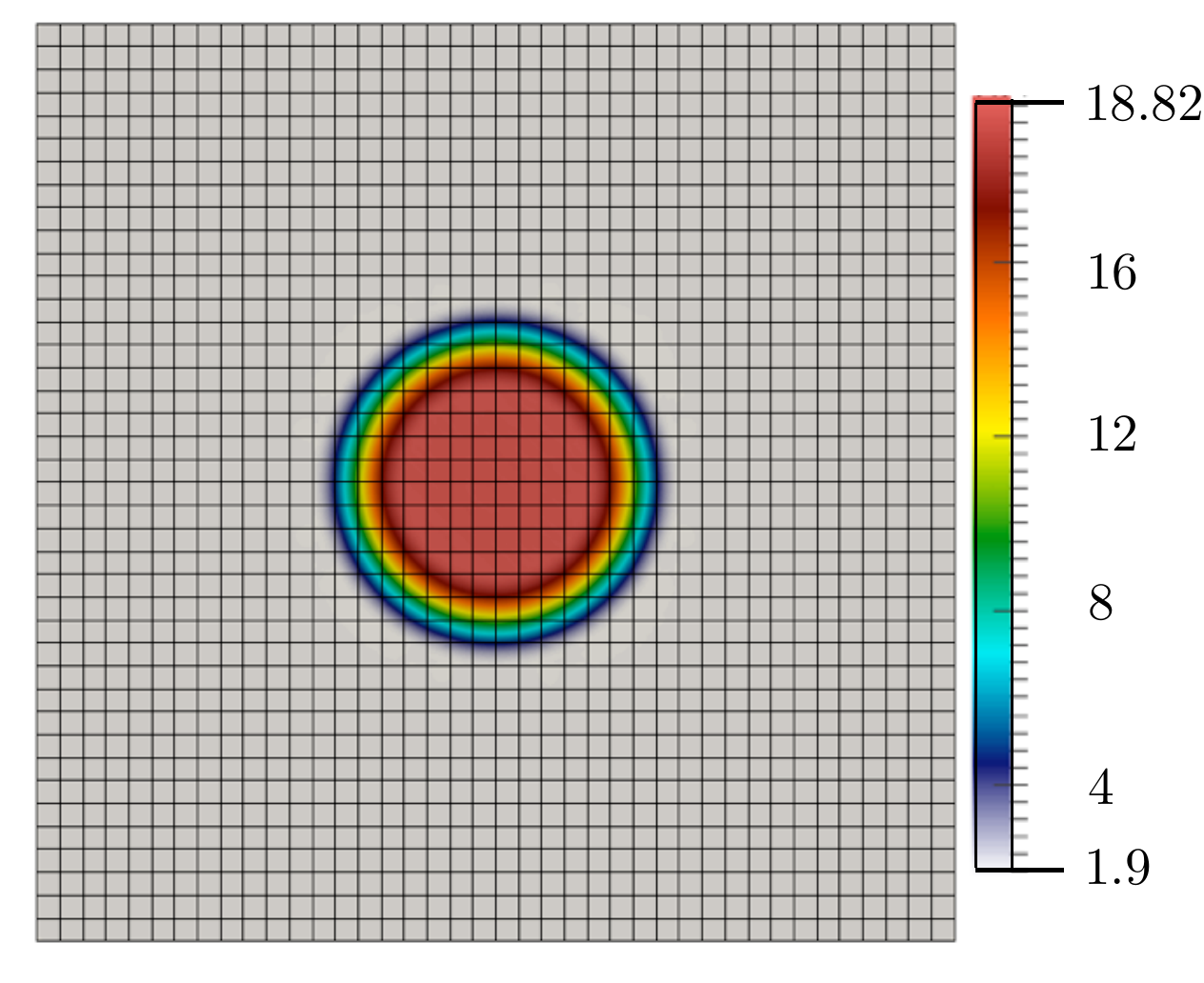}}\hspace{0mm}
\subfloat[$k=6$]{\includegraphics[height=0.137\textheight]{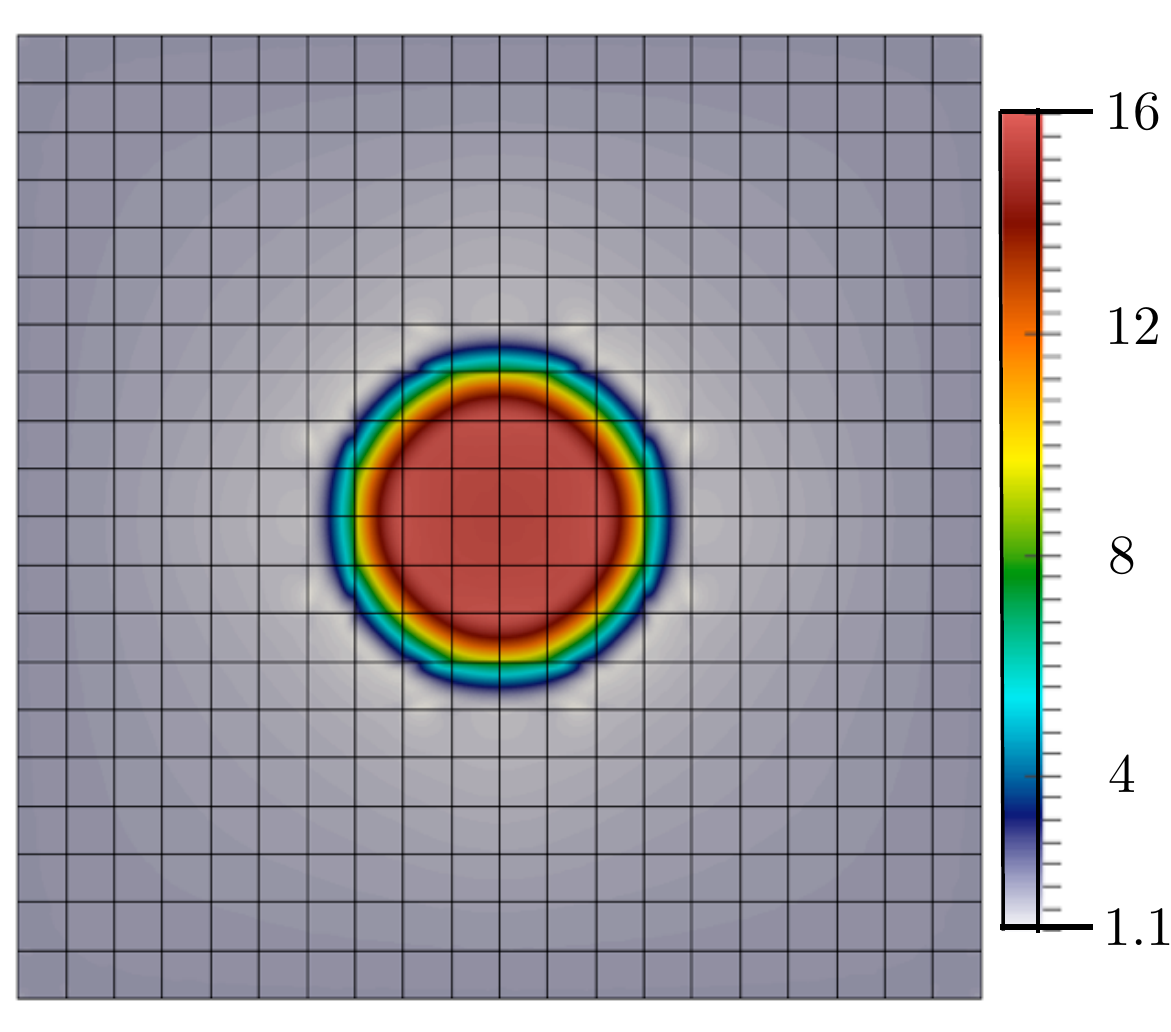}}
\caption{Reconstruction $\beta_\delta(\bx)$ for different values of $h = k\, h_0$ with $\delta = 10^{-7}$.}
\label{fig_h_varies}
\end{figure}

\begin{figure}[b]	
\centering
\includegraphics[width=0.55\textwidth]{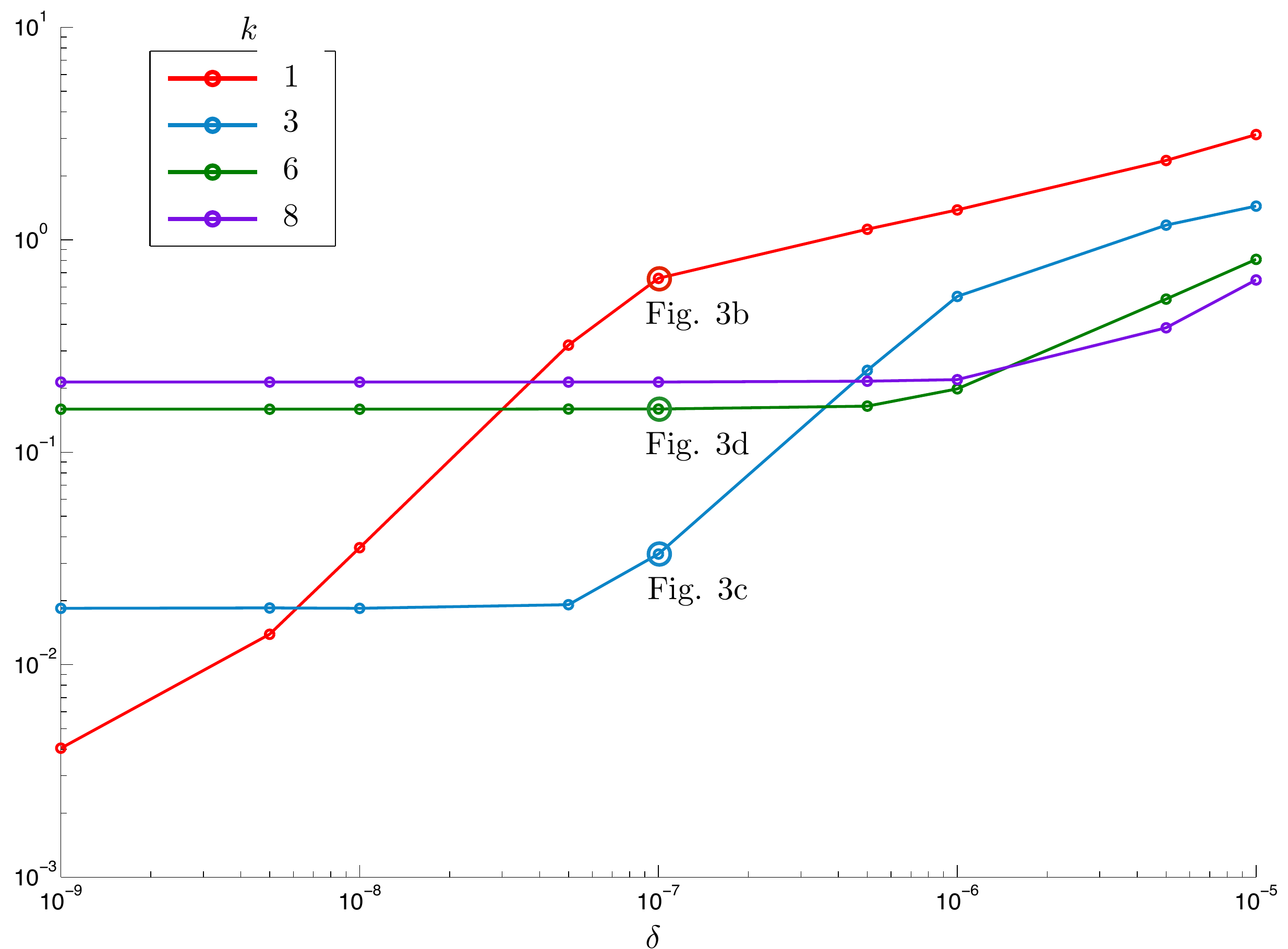}\vspace{-4mm}
\caption{Relative error in $H^1(\Omega)$-norm on reconstructed parameters $(\alpha,\beta)$ w.r.t. noise value $\delta$ for different values of $h= k\, h_0$. The doubly-circled dots correspond to Figure \ref{fig_h_varies}.}
	\label{fig_delta_varies}
\end{figure}

\textit{Convergence in $\delta$.}  The convergence of the reconstruction error with respect to the noise level $\delta$ is now investigated. In Figure \ref{fig_delta_varies}, the relative reconstruction errors in $H^1(\Omega)$-norm are compared for different values $h$. Two expected trends are highlighted. Firstly, at a given noise value $\delta$ there exists an optimal mesh size parameter $h$, for the projection and differentiation steps, for which the reconstruction is the best in $H^1(\Omega)$-norm, as shown previously. Secondly, for a fixed $h$, when the noise level decreases in $L^\infty(\Omega)$-norm the reconstruction quality reaches a plateau. This is due to a loss of resolution associated with the projection on a coarse mesh. These numerical results essentially show that the proposed algorithm achieves the standard trade-off between regularization and resolution.

\paragraph{Frequency dependent case.} In light of formulae \eref{eq:classical_def} and \eref{eq:def_E}, it is clear that using oscillating elasticity solutions, with locally vanishing gradients, might be detrimental to the invertibility condition of Assumptions \ref{mega_hyp}. This implies that low frequency configurations should be preferred so that we set $\omega_1 = 1$ and $\omega_2 = 0$. The boundary conditions are prescribed as
\[
	\bg_1(\bx) = 1+( \bx \cdot \be_1 ) \, \be_1+( \bx \cdot \be_2 ) \, \be_2, \quad \bg_2(\bx) = (1+ \bx \cdot \be_1 + \bx \cdot \be_2) (\be_1 - \be_2).
\]
The exact moduli distributions are computed using \eref{eq:formulae_alpha_beta} and displayed in Fig. \ref{exact:mod}. The modulus of the corresponding elasticity solutions and $\det\Escr$ are plotted in Figure \ref{fig_solution_frequency}. 

\begin{figure}[t]	
\centering
\subfloat[$|\bu_1(\bx)|$]{\includegraphics[height=0.15\textheight]{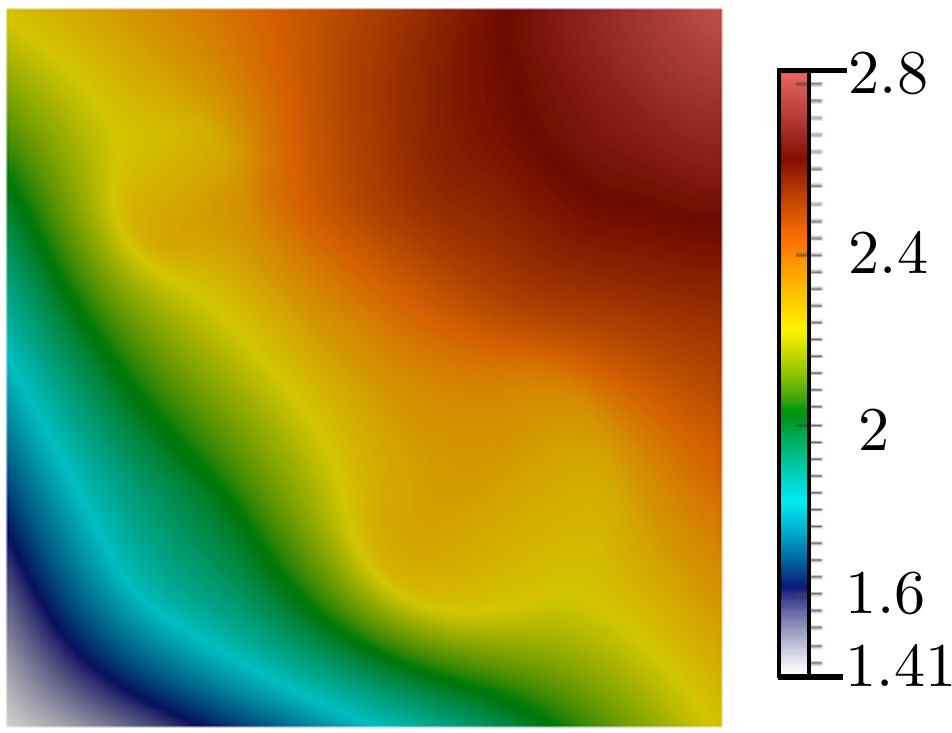}}\hspace{2mm}
\subfloat[$|\bu_2(\bx)|$]{\includegraphics[height=0.15\textheight]{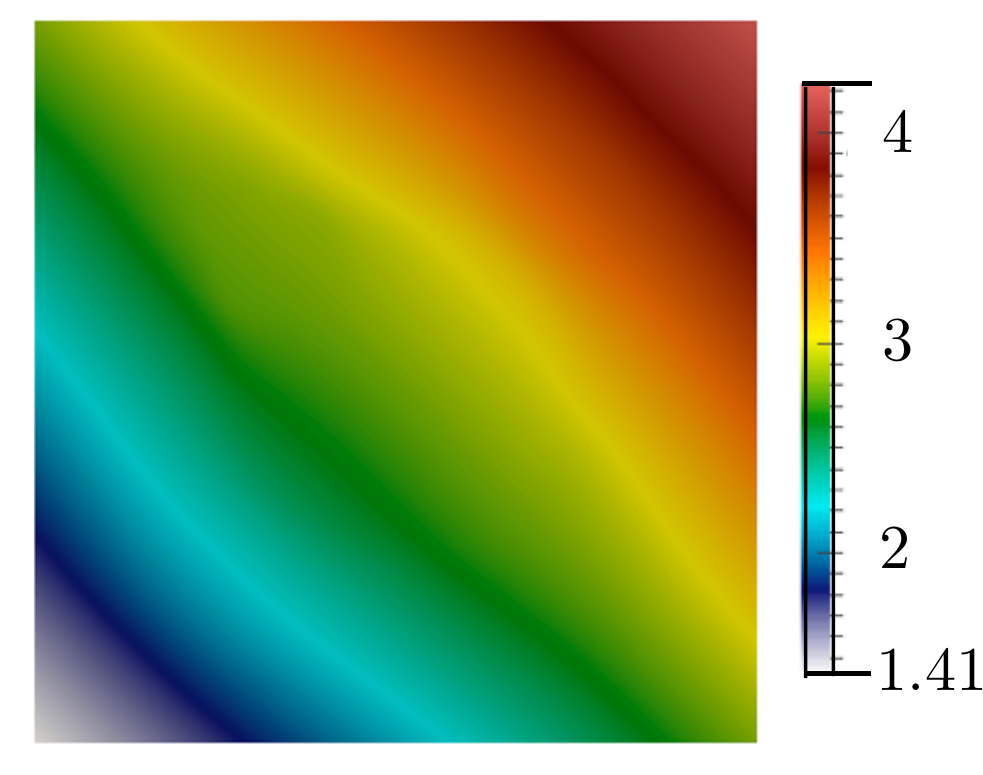}}\hspace{2mm}
\subfloat[$\det\Escr(\bx)$]{\includegraphics[height=0.15\textheight]{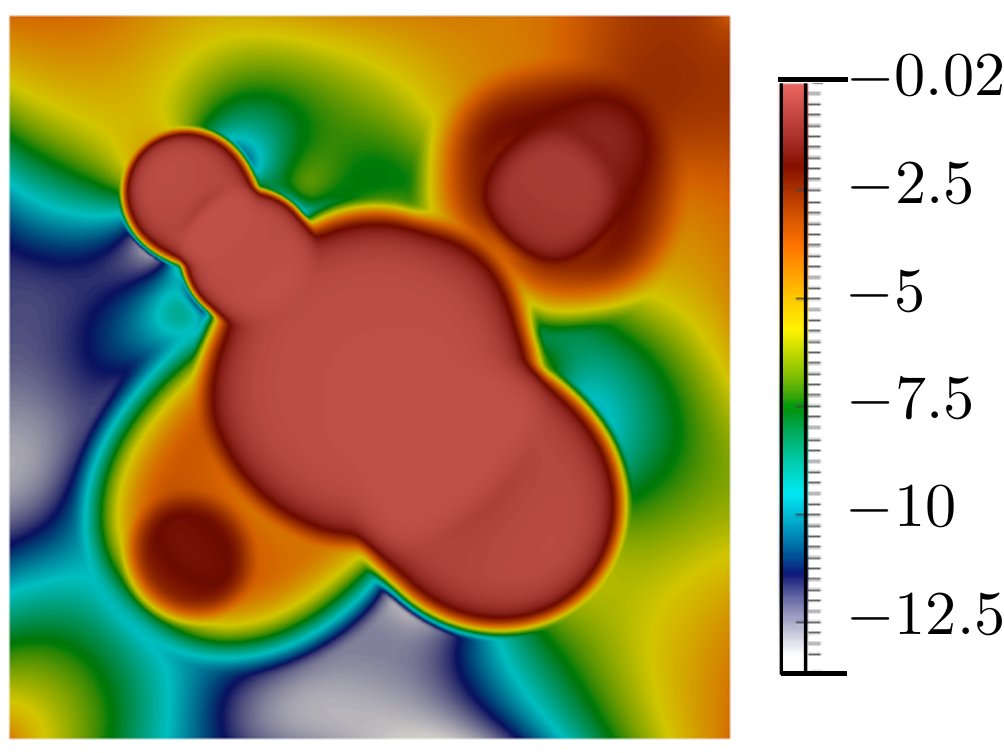}}
\caption{Computed elasticity solutions in the frequency dependent case}
\label{fig_solution_frequency}
\end{figure}

Figure \ref{fig_reconstruction_frequency} shows three reconstructions. Fig. \ref{reconst:mod} corresponds to a noise-free configuration ($\delta = 0$) and no regularization ($h=h_0=1/120$). The associated relative reconstruction error in $H^1(\Omega)$-norm is $0.0033$. Figure \ref{reconst:mod:no} corresponds to $\delta = 10^{-7}$ and mesh size $h=1/120$, i.e. a reconstruction without regularization. The corresponding relative error is $0.83$ and the reconstruction of $\alpha(\bx)$, $\beta(\bx)$ is poor. Finally, for this noise value but with regularization, i.e. $h=1/24$, then the relative error decreases to $0.67$ and the reconstruction is qualitatively improved in terms of identification of the number of heterogeneities, their locations and relative sizes, see Fig. \ref{reconst:mod:no:reg}. 

\begin{figure}[th!]	
\centering
\subfloat[$\big(\alpha(\bx), \beta(\bx) \big)$]{\label{exact:mod}\includegraphics[height=0.3\textheight]{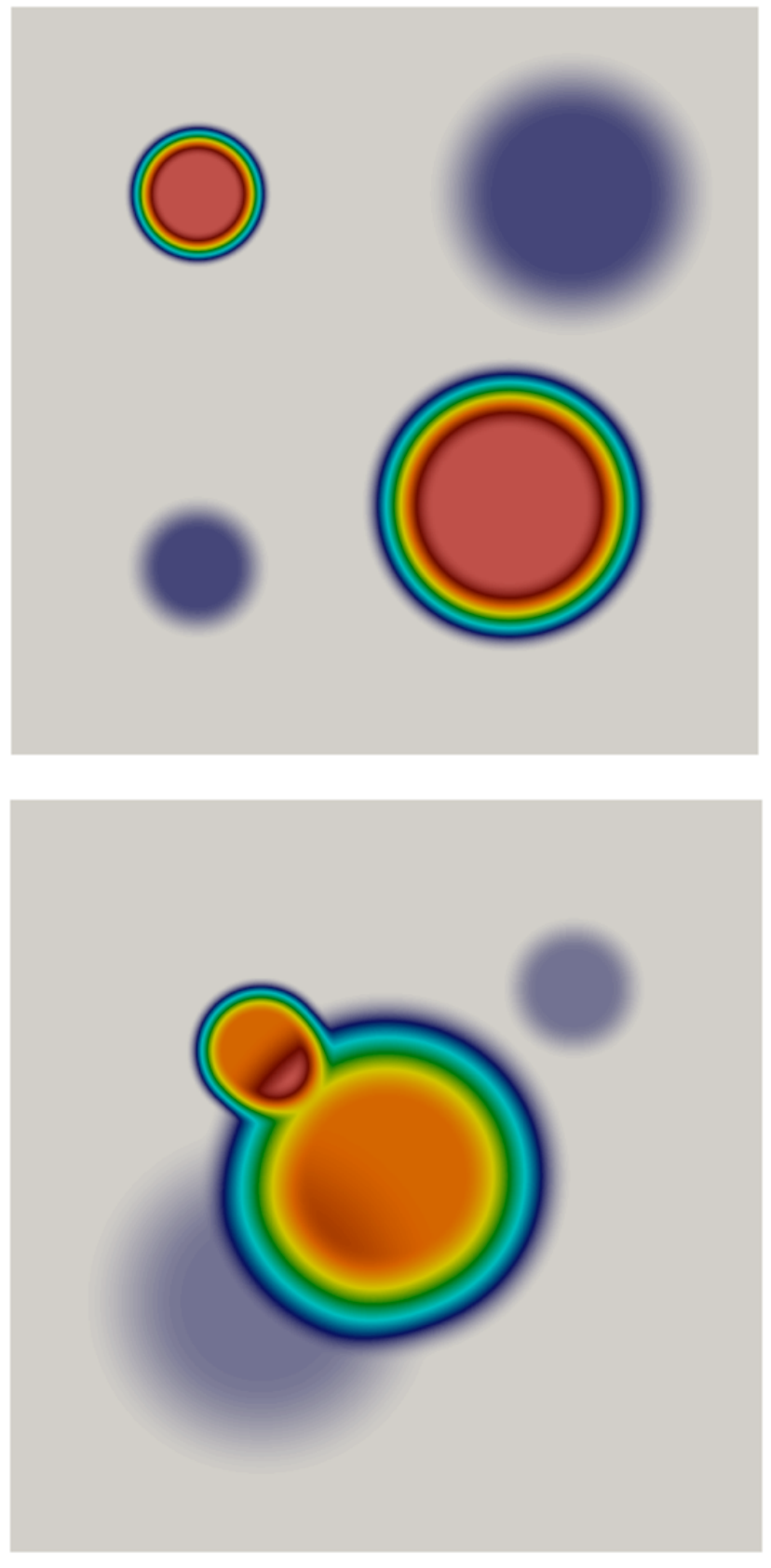}}\hspace{0mm}
\subfloat[$\delta = 0, k=1$]{\label{reconst:mod}\includegraphics[height=0.3\textheight]{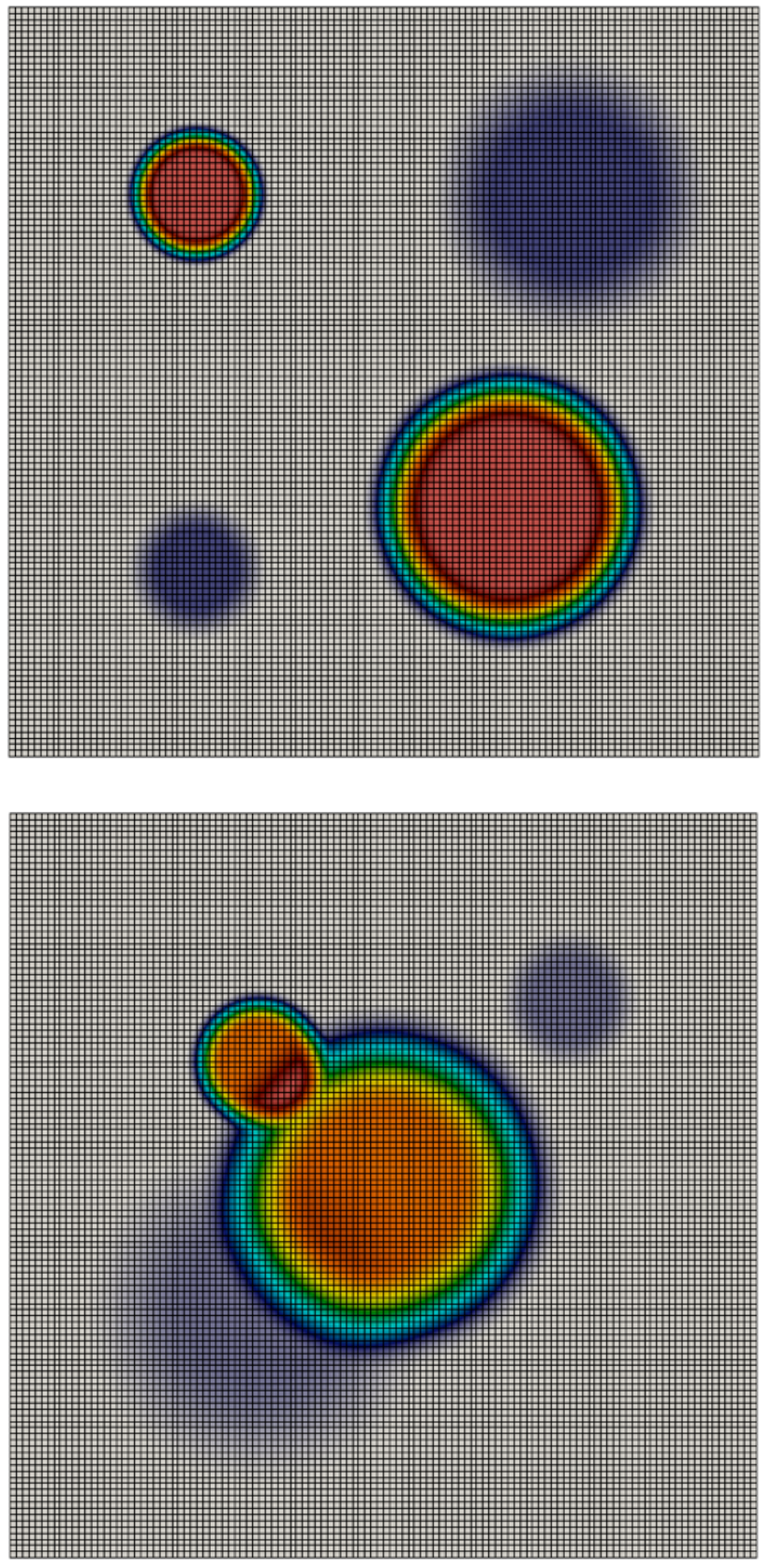}}\hspace{0mm}
\subfloat[$\delta = 10^{-7}, k=1$]{\label{reconst:mod:no}\includegraphics[height=0.3\textheight]{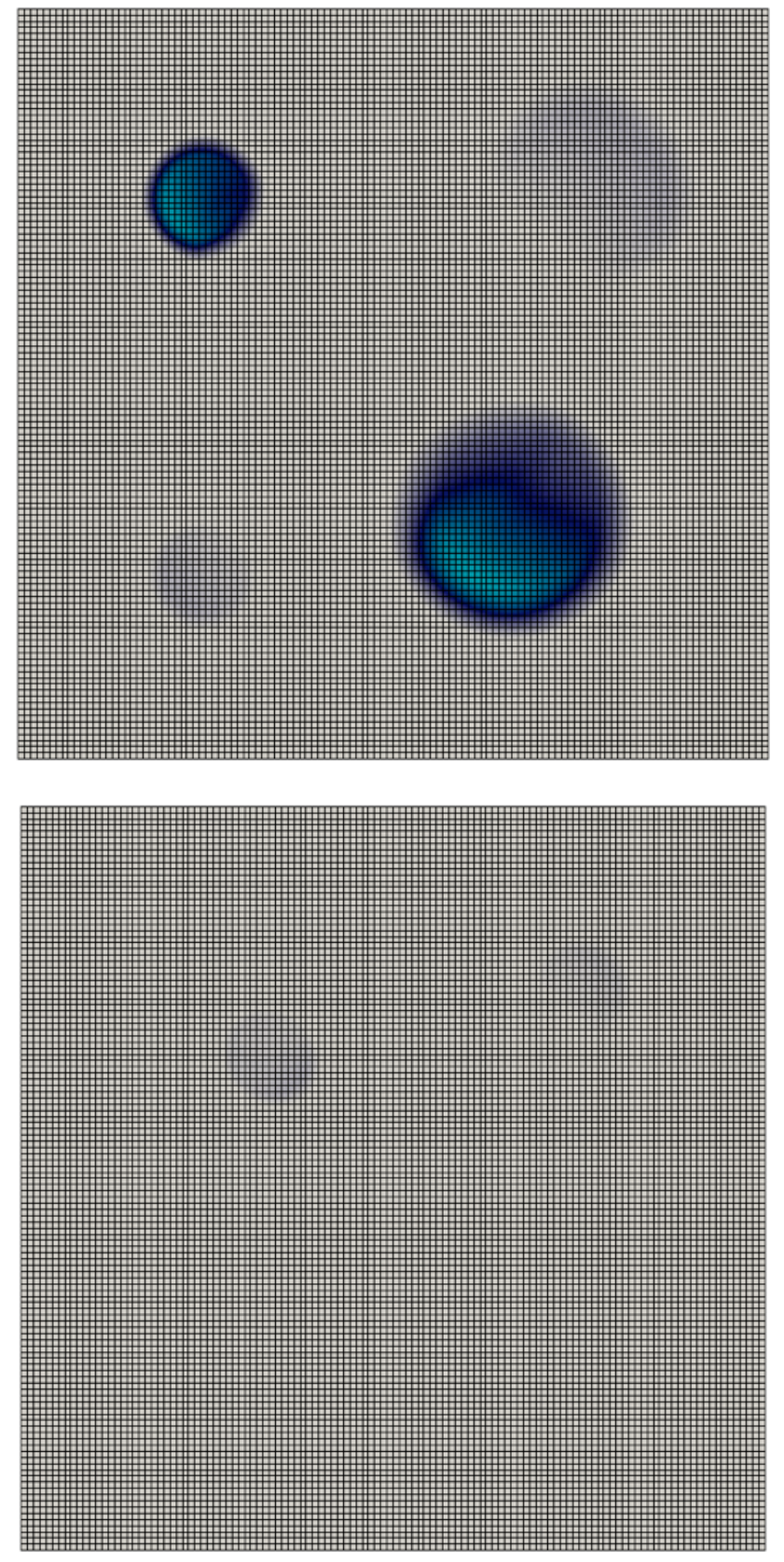}}\hspace{0mm}
\subfloat[$\delta = 10^{-7}, k=5$]{\label{reconst:mod:no:reg}\includegraphics[height=0.3\textheight]{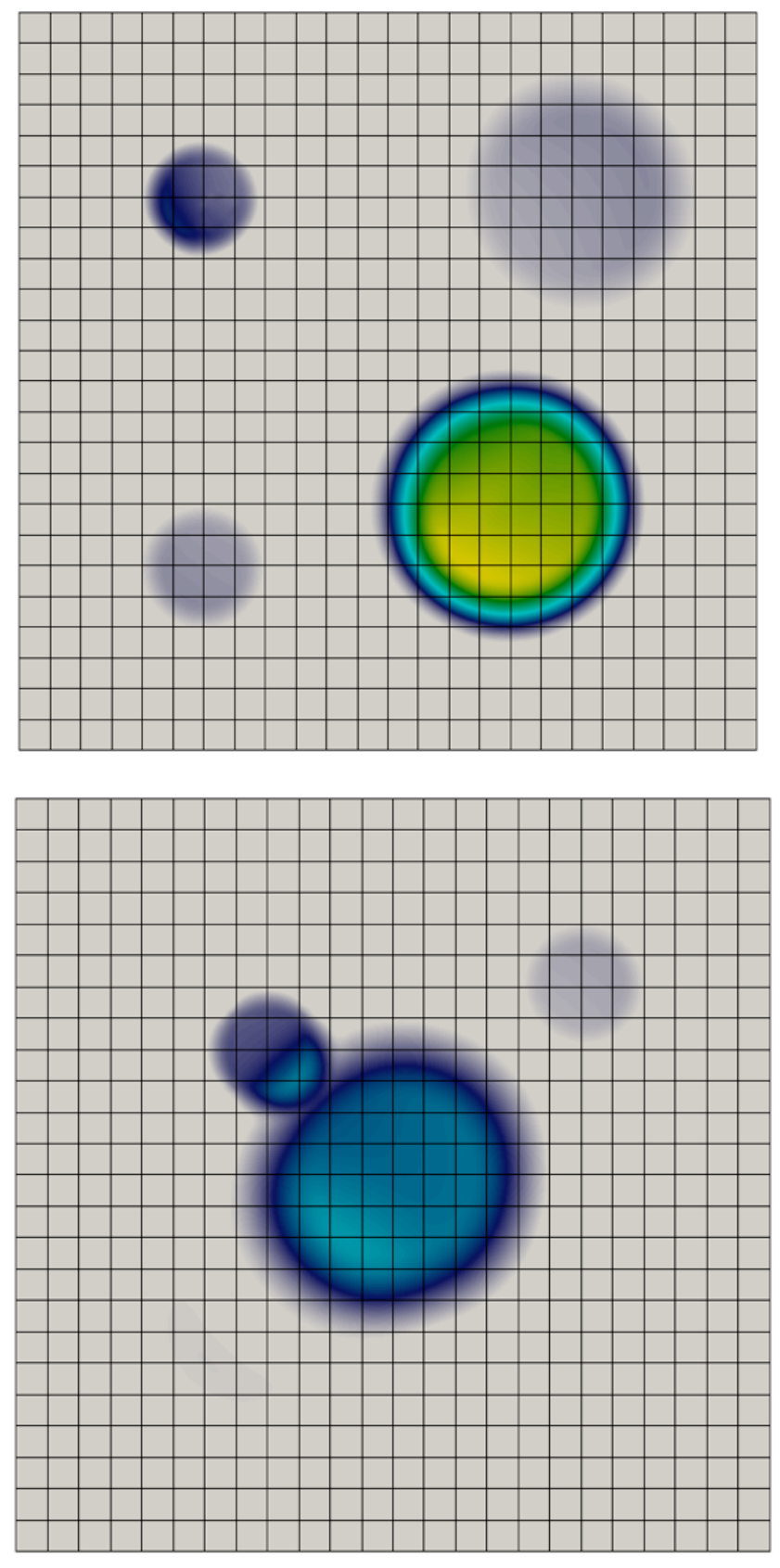}}
\subfloat{\includegraphics[height=0.3\textheight]{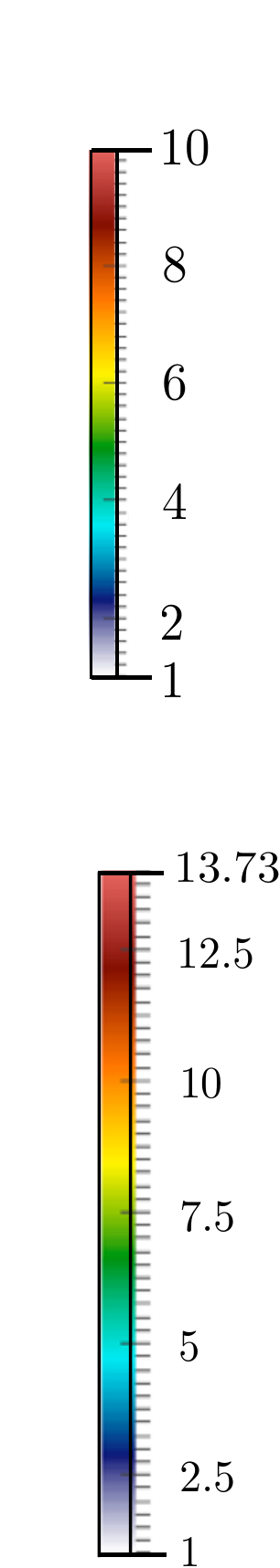}}
\caption{(a) Exact values of $\alpha(\bx)$ \emph{(top)} and $\beta(\bx)$ \emph{(bottom)}; Corresponding reconstructions with (b) no noise nor regularization, (c) with noise but no regularization, (d) with noise and regularization.}
\label{fig_reconstruction_frequency}
\end{figure}

\begin{figure}[bht]	
\centering
\subfloat[$\big(\alpha(\bx), \beta(\bx) \big)$]{\label{exact:rand}\includegraphics[height=0.3\textheight]{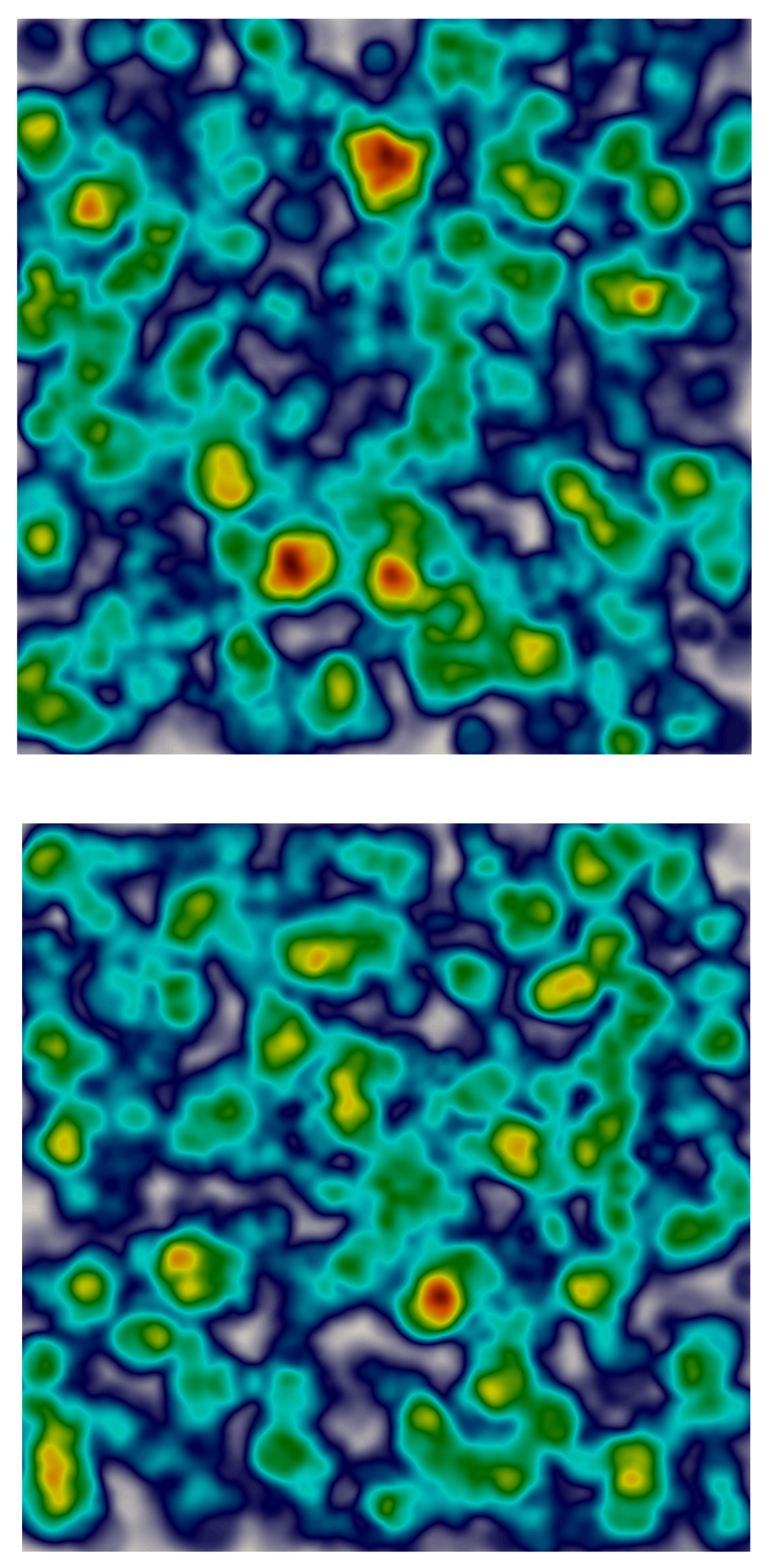}}\hspace{-1mm}
\subfloat[$\delta = 0, k=1$]{\label{reconst:rand}\includegraphics[height=0.3\textheight]{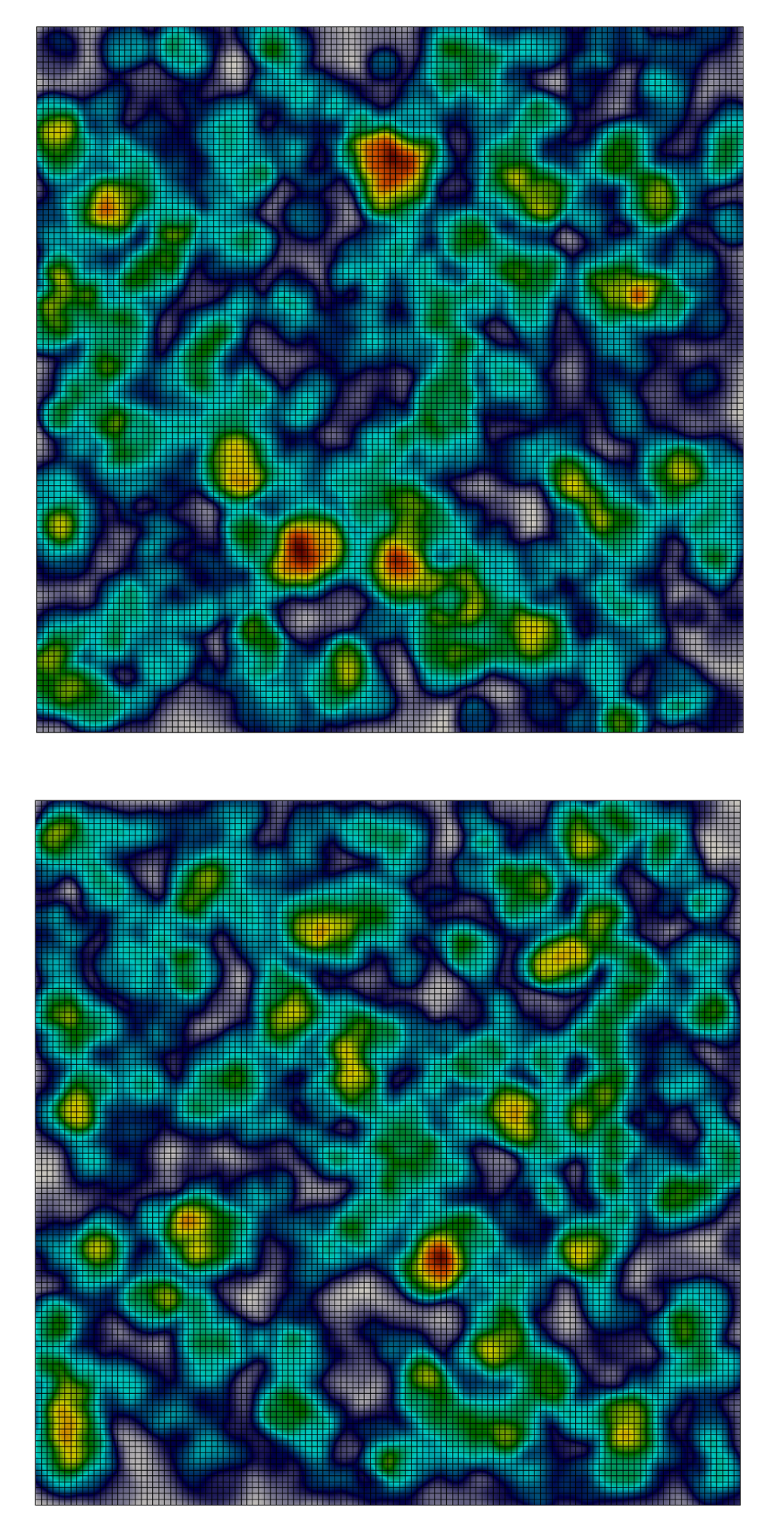}}\hspace{-2mm}
\subfloat[$\delta = 10^{-6}, k=1$]{\label{reconst:rand:no}\includegraphics[height=0.3\textheight]{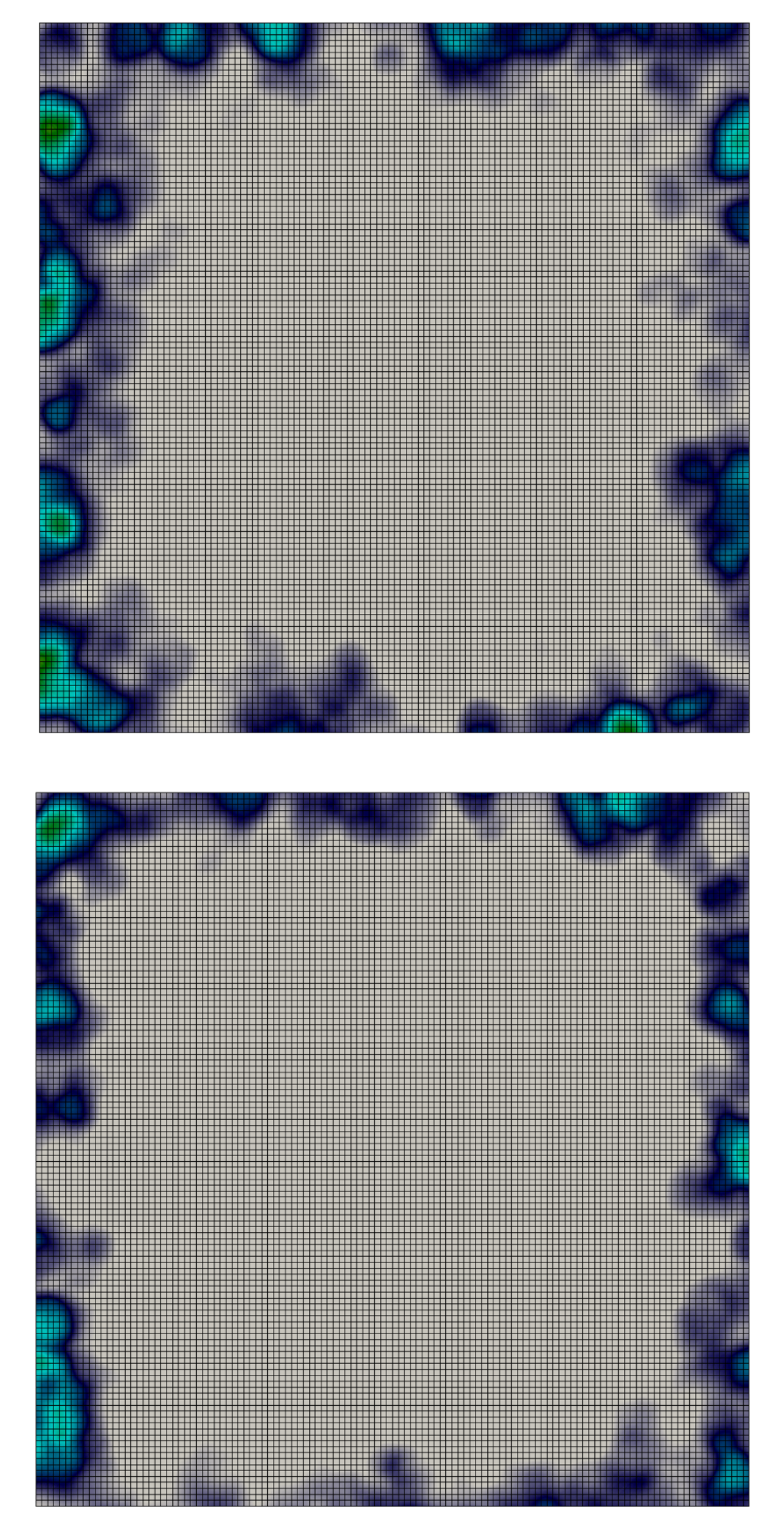}}\hspace{-2mm}
\subfloat[$\delta = 10^{-6}, k=5$]{\label{reconst:rand:no:reg}\includegraphics[height=0.3\textheight]{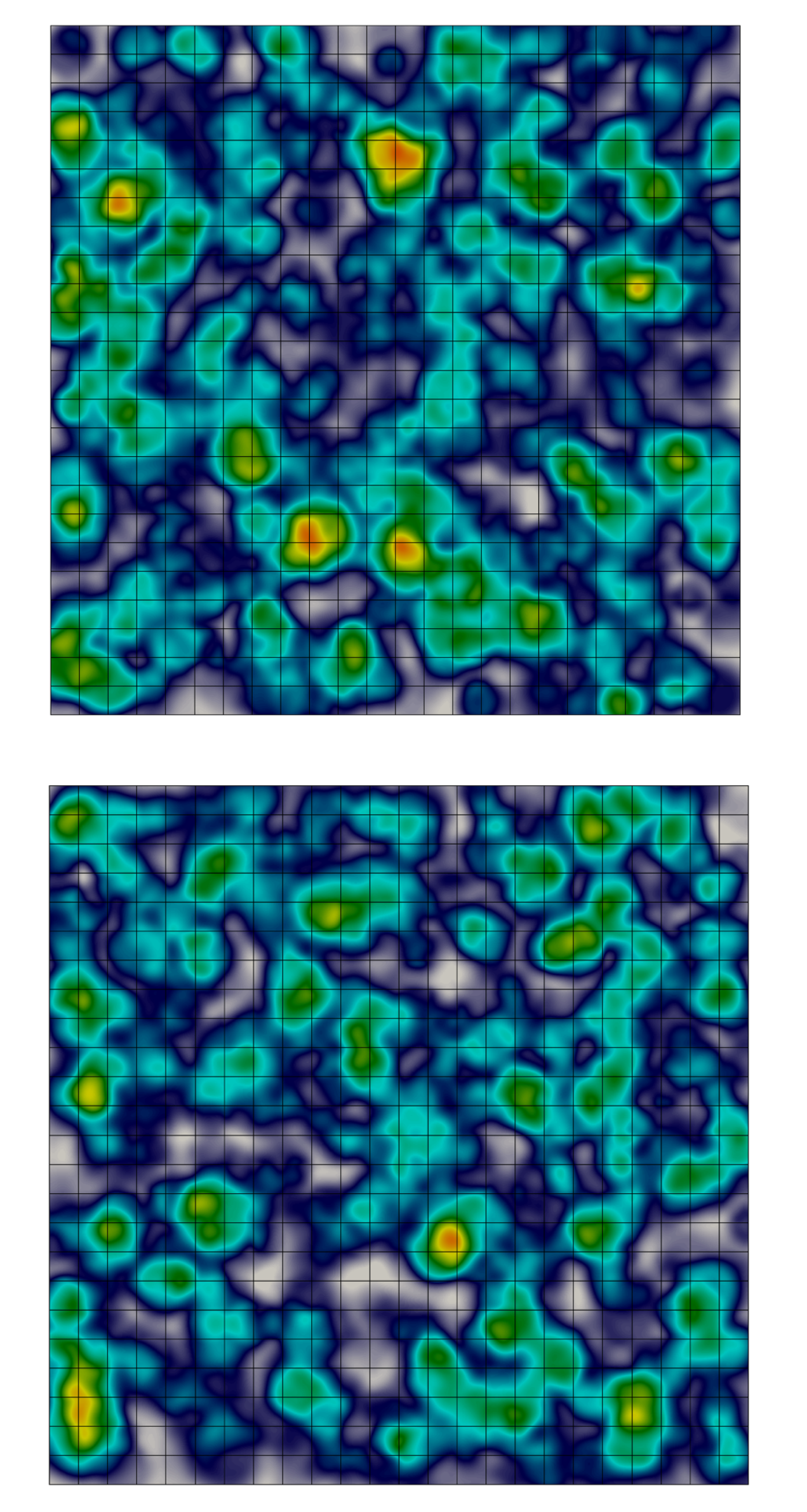}}
\subfloat{\includegraphics[height=0.3\textheight]{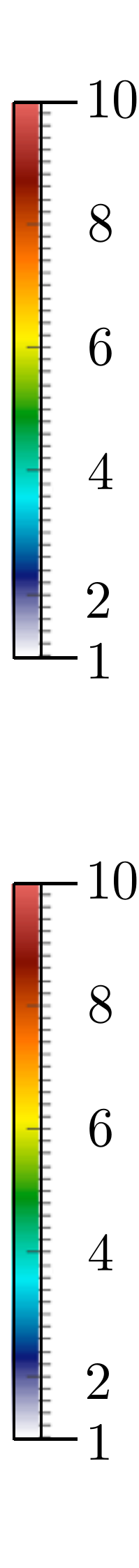}} 
\caption{(a) Exact values of $\alpha(\bx)$ \emph{(top)} and $\beta(\bx)$ \emph{(bottom)}; Corresponding reconstructions with (b) no noise nor regularization, (c) with noise but no regularization, (d) with noise and regularization.}
\label{fig_reconstruction_random}
\end{figure}

\paragraph{Random parameters distributions} This section is concluded with the investigation of a case of randomly varying moduli. Their studied exact distributions in Fig. \ref{exact:rand} are constructed from \eref{eq:formulae_alpha_beta}, using $N_\alpha=N_\beta=1000$ and random parameters. They are characterized by a correlation length much smaller than the domain size in order to give rise to oscillating strain field solutions. The settings are these of the previous paragraph but $\omega_1=\omega_2=0$. The reconstruction in Fig. \ref{reconst:rand} from noise-free measurements is fairly good, with a relative error of $0.0067$. A noise value $\delta=10^{-6}$ makes the non-regularized reconstruction ineffective: Fig. \ref{reconst:rand:no}, relative error $0.87$. Alternatively, the recovery of the unknowns is remarkably improved using the proposed regularization method as shown in Fig. \ref{reconst:rand:no:reg} with a relative reconstruction error of $0.20$.

\section{Conclusion}

An algorithm for the quantitative reconstruction of constitutive elasticity parameters has been investigated. It is based on the construction and inversion of a linear operator computed from full-field internal measurements, possibly noisy, of elasticity solutions.

The highlights of this study are: 
\begin{enumerate}[leftmargin=0cm,itemindent=8mm,labelwidth=\itemindent,labelsep=0cm,align=left,label=(\roman*),noitemsep,topsep=0pt] 
	\item Construction of a linear operator from available full-field measurements which in turn are required to satisfy some invertibility and compatibility conditions.
	\item Characterization of this symmetric definite positive linear operator and derivation of an associated variational formulation. 
	\item Investigation of a numerical differentiation scheme based on regularizations by projections on coarse meshes but within a high-order finite element formulation.
	\item Theoretical results proving convergence of the algorithm when noise in the data vanishes in the $L^\infty(\Omega)$-norm yet not necessarily in the $H^2(\Omega)$-norm.
	\item The proposed algorithm is theoretically valid at any frequency. Invertibility and compatibility conditions make it applicable at least to low-frequency configurations.
	\item Numerical results highlighting the method's potential for practical applications.
\end{enumerate}

Looking forward, it would be relevant to address the cases where the invertibility or compatibility conditions might locally be not satisfied, e.g. when dealing with highly oscillating measurements. A possible strategy is to work with a larger set of data so that inversion formula \eref{sys:op:ref} can be constructed pointwise using the pair of solutions maximizing $| \det \Escr \,  |$ locally. Moreover, the proposed approach might be extended to the case of non-complete data, where either internal measurements or boundary conditions are partially available. Finally, while the approach finds direct applications in the field of non-destructive material testing, its extension to the model of nearly incompressible solids used in medical imaging is the subject of ongoing research.

\vspace{-3mm}

\section*{References}\vspace{-3mm}
\bibliographystyle{unsrt}
\bibliography{biblio}

\end{document}